\documentclass[11pt,twoside,a4paper]{article}
\usepackage{amssymb}
\usepackage{amsmath}
\usepackage{amsthm}
\usepackage{color}
\usepackage{mathrsfs}

\usepackage{amsfonts}
\usepackage{amssymb}
\usepackage{amsmath}
\usepackage{amsthm}
\usepackage{bbm}
\usepackage{verbatim}
\usepackage{url}

\allowdisplaybreaks

\def\fz{\infty}

\def\supp{{\mathop\mathrm{\,supp\,}}}

\def\ls{\lesssim}
\def\gs{\gtrsim}

\def\dint{\displaystyle\int}

\def\dfrac{\displaystyle\frac}

\def\r{\right}
\def\lf{\left}

\newtheorem{thm}{Theorem}[section]
\newtheorem{lem}[thm]{Lemma}
\newtheorem{defn}[thm]{Definition}

\newcommand{\intav}{-\!\!\!\!\!\!\int}

\numberwithin{equation}{section}
\begin{document}
\arraycolsep=1pt

\title{\large\bf  A NOTE ON COMMUTATORS ON WEIGHTED MORREY SPACES \\ ON SPACES OF HOMOGENEOUS TYPE}
\author{Ruming Gong, Ji Li, Elodie Pozzi and Manasa N. Vempati}

\date{}
\maketitle

\begin{center}
\begin{minipage}{13.5cm}\small

{\noindent  {\bf Abstract:}\
In this paper we study the boundedness and compactness characterizations of the commutator of  Calder\'{o}n--Zygmund operators $T$  on spaces of homogeneous type $(X,d,\mu)$ in the sense of Coifman and Weiss. More precisely,  We show that the commutator $[b, T]$  is bounded on weighted Morrey space $L_{\omega}^{p,\kappa}(X)$ ($\kappa\in(0,1), \omega\in A_{p}(X), 1<p<\infty$)  if {\color{black}and only if}\ $b$ is in the BMO space. Moreover, the commutator $[b, T]$  is compact on weighted Morrey space $L_{\omega}^{p,\kappa}(X)$ ($\kappa\in(0,1), \omega\in A_{p}(X), 1<p<\infty$)  if {\color{black}and only if}\ $b$ is in the VMO space.
}

\end{minipage}
\end{center}

\bigskip


\medskip


\section{Introduction}
\setcounter{equation}{0}

\smallskip

It is well-known that the boundedness and compactness of Calder\'on--Zygmund operator commutators on certain function spaces and their characterizations play an important role in various area, such as harmonic analysis, complex analysis, (nonlinear) PDE, etc. See for example \cite{CRW,CLMS,B,HLW, HPW, DLLW,Hyt, GLW, LOR, LOR2, ATC} and the references therein.
Recently,  equivalent characterizations
of the boundedness and the compactness of commutators were further extended to
Morrey spaces over the Euclidean space  by Di Fazio and Ragusa \cite{DiFazioRagusa91BUMIA} and
Chen et al.\,\cite{CDW12CJM},  and to weighted Morrey spaces by Komori and Shirai \cite{KS} for Calder\'on--Zygmund operator commutators and by Tao, Da. Yang and Do. Yang \,\cite{TYY,TYY2} for the Cauchy integral and Buerling-Ahlfors transformation commutator, respectively. For more results on the boundedness of operators on Morrey spaces in different settings, we refer the reader to other studies \cite{AX,FLY,MWY,FGPW}.

 Thus, along this literature,  it is natural to study the boundedness and compactness of Calder\'on--Zygmund operator commutators on weighted Morrey spaces in a more general setting:  spaces of homogeneous type in the sense of Coifman and Weiss \cite{cw77},
as Yves Meyer remarked in his preface to \cite{DH}, \textquotedblleft One is
amazed by the dramatic changes that occurred in analysis during the
twentieth century. In the 1930s complex methods and Fourier series played a
seminal role. After many improvements, mostly achieved by the Calder\'on--Zygmund school, the action takes place today on
spaces of homogeneous type. No group structure is available, the Fourier
transform is missing, but a version of harmonic analysis is still present.
Indeed the geometry is conducting the analysis.\textquotedblright

We say
that $(X,d,\mu)$ is a {space of homogeneous type} in the
sense of Coifman and Weiss if $d$ is a quasi-metric on~$X$
and $\mu$ is a nonzero measure satisfying the doubling
condition. A \emph{quasi-metric}~$d$ on a set~$X$ is a
function $d: X\times X\longrightarrow[0,\infty)$ satisfying
(i) $d(x,y) = d(y,x) \geq 0$ for all $x$, $y\in X$; (ii)
$d(x,y) = 0$ if and only if $x = y$; and (iii) the
\emph{quasi-triangle inequality}: there is a constant $A_0\in
[1,\infty)$ such that for all $x$, $y$, $z\in X$, 
\begin{eqnarray}\label{eqn:quasitriangleineq}
    d(x,y)
    \leq A_0 [d(x,z) + d(z,y)].
\end{eqnarray}
We say that a nonzero measure $\mu$ satisfies the
\emph{doubling condition} if there is a constant $C_\mu$ such
that for all $x\in X$ and $r > 0$,
\begin{eqnarray}\label{doubling condition}
   \mu(B(x,2r))
   \leq C_\mu \mu(B(x,r))
   < \infty,
\end{eqnarray}
where $B(x,r)$ is the quasi-metric ball by $B(x,r) := \{y\in X: d(x,y)
< r\}$ for $x\in X$ and $r > 0$.
We point out that the doubling condition (\ref{doubling
condition}) implies that there exists a positive constant
$n$ (the \emph{upper dimension} of~$\mu$)  such
that for all $x\in X$, $\lambda\geq 1$ and $r > 0$,
\begin{eqnarray}\label{upper dimension}
    \mu(B(x, \lambda r))
    \leq  C_\mu\lambda^{n} \mu(B(x,r)).
\end{eqnarray}
Throughout this paper we assume that $\mu(X)=\infty$ and that $\mu(\{x_0\})=0$ for every $x_0\in X$.

We now recall the definition of Calder\'{o}n--Zygmund operators on
spaces of homogeneous type.

\begin{defn}
\label{def 1} We say that $T$ is a Calder\'{o}n--Zygmund operator on $%
(X,d,\mu )$ if $T$ is bounded on $L^{2}(X)$ and has an associated kernel $%
K(x,y)$ such that $T(f)(x)=\int_{X}K(x,y)f(y)d\mu (y)$ for any $x\not\in
\mathrm{supp}\,f$, and $K(x,y)$ satisfies the following estimates: for all $%
x\not=y$,
\begin{equation}
|K(x,y)|\leq {\frac{{C}}{{V(x,y)}}},  \label{size of C-Z-S-I-O}
\end{equation}%
and for $d(x,x^{\prime })\leq (2A_{0})^{-1}d(x,y)$,
\begin{equation}
|K(x,y)-K(x^{\prime },y)|+|K(y,x)- K(y,x^{\prime })|\leq {\frac{C}{V(x,y)}}%
\beta \left( {\frac{d(x,x^{\prime })}{d(x,y)}}\right) ,
\label{smooth of C-Z-S-I-O}
\end{equation}%
where $V(x,y)=\mu (B(x,d(x,y)))$, $\beta :[0,1]\rightarrow \lbrack 0,\infty
)$ is continuous, increasing, subadditive, and $\omega (0)=0$. Throughout this paper we assume that $\beta(t)= t^{\sigma_0}$, for some $\sigma_0>0$.
\end{defn}
Note that by the doubling condition we have that $V(x,y)\approx V(y,x)$.  From \cite{Dgk} we assume for any Calder\'{o}n--Zygmund operator $T$ as in Definition \ref{def 1} with $\beta(t)\rightarrow 0$ as $t\rightarrow 0$, the following ``non-degenerate" condition holds:

There exists positive constant $c_o$ and $\bar{A}$
such that for every $x\in X$ and $r>0$, there exists $y\in B(x,\bar{A}r)\setminus B(x,r)$, satisfying

\begin{equation}\label{lowerbound}
    |K(x,y)|\geq \frac{1}{c_0\mu(B(x,r))}.
\end{equation}

This condition gives a lower bound on the kernel and in $\mathbb{R}^n$ this ``non degenerate" condition was first proposed in \cite{Hyt}.
On stratified Lie groups, a similar condition of the Riesz transform kernel  lower bound  was verified in \cite{DLLW}.

Let $T$ be a Calder\'on--Zygmund operator on $X$.
Suppose $b\in L^1_{\rm loc}(X)$ and $f\in L^p(X)$. Let $[b, T]$ be the commutator defined by
\begin{equation*}
[b, T]f(x):= b(x)T( f)(x)-T(bf)(x).
\end{equation*}

Let $p\in(1,\infty),\kappa\in(0,1)$ and $\omega\in A_p(X)$. The \emph{weighted Morrey space} $L_{\omega}^{p,\kappa}(X)$ is defined by
\begin{equation*}
   L_{\omega}^{p,\kappa}(X) := \{f\in L_{loc}^{p}(X):\|f\|_{L_{\omega}^{p,\kappa}(X)} <\infty\}
\end{equation*}
Here
\begin{equation*}
    \|f\|_{L_{\omega}^{p,\kappa}(X)} := \sup_{B}\left\{\frac{1}{\omega(B)^\kappa}\int_{B}|f(x)|^p \omega(x)d\mu(x)\right\}^\frac{1}{p}.
\end{equation*}

Our  main results are the following theorems.

\begin{thm}\label{thm main1}
Let $p\in(1,\infty), \kappa\in(0,1)$ and $ \omega\in A_{p}(X)$. Suppose $b\in L^1_{\rm loc}(X)$ and that $T$ is a Calder\'on--Zygmund operator as in Definition \ref{def 1} and satisfies the non-degenerate condition \eqref{lowerbound}. Then the commutator $[b,T]$  has the following boundedness characterization:
\begin{enumerate}
    \item[\rm (i)] If $b\in BMO(X)$, then $[b,T]$ is bounded on
    $L_{\omega}^{p,\kappa}(X)$.
    \item[\rm(ii)] If b is real valued and $[b,T]$ is bounded on $L_{\omega}^{p,\kappa}(X)$, then $b\in BMO(X)$.
\end{enumerate}
\end{thm}
\begin{thm}\label{thm main2}

Let $p\in(1,\infty), \kappa\in(0,1)$ and $ \omega\in A_{p}(X)$. Suppose $b\in L^1_{\rm loc}(X)$ and that $T$ is a Calder\'on--Zygmund operator as in Definition \ref{def 1} and satisfies the non-degenerate condition \eqref{lowerbound}. Then the commutator $[b,T]$  has the following  compactness characterization:
\begin{enumerate}
    \item[\rm(i)] If $b\in VMO(X)$, then $[b,T]$ is compact on
    $L_{\omega}^{p,\kappa}(X)$.
    \item[\rm(ii)] If b is real valued and $[b,T]$ is compact on $L_{\omega}^{p,\kappa}(X)$, then $b\in VMO(X)$.
\end{enumerate}

\end{thm}

\medskip
Throughout the paper,
we denote by $C$ and $\widetilde{C}$ {\it positive constants} which
are independent of the main parameters, but they may vary from line to
line. For every $p\in(1, \fz)$, we denote by $p'$ the conjugate of $p$, i.e., $\frac{1}{p'}+\frac{1}{p}=1$.  If $f\le Cg$ or $f\ge Cg$, we then write $f\ls g$ or $f\gs g$;
and if $f \ls g\ls f$, we  write $f\approx g.$

\section{Preliminaries on Spaces of Homogeneous Type}
\label{s2}
\noindent

Let $(X,d,\mu)$ be a space of homogeneous type as mentioned in Section 1. We now recall the BMO and VMO space.

\begin{defn}\label{d-bmo}
A function $b\in L^1_{\rm loc}(X)$ belongs to
the  BMO space $BMO(X)$ if
\begin{equation*}
\|b\|_{BMO(X)}:=\sup_{B}M(b,B):=\sup_{B}{1\over \mu(B)}\dint_{B}
\lf|b(x)-b_{B}\r|\, d\mu(x)<\fz,
\end{equation*}
where the sup is taken over all quasi-metric balls $B\subset X$ and
$$ b_B= {1\over\mu(B)} \int_B b(y)d\mu(y). $$
\end{defn}

The following John-Nirenberg inequalities on spaces of homogeneous type comes from \cite{Kr}.

\begin{lem}[\cite{Kr}]\label{lem-jn1}
If $f\in {\rm BMO}(X)$, then there exist positive constants $C_1$ and $C_2$ such that for every ball $B\subset X$ and every $\alpha>0$, we have
$$\mu(\{x\in B: |f(x)-f_B|>\alpha \})\leq C_1\lambda(B)\exp\Big\{- {C_2\over \|f\|_{{\rm BMO}(X)}}\alpha\Big\}.$$
\end{lem}
We recall the median value $\alpha_B(f)$ (\cite{Pxq}). For any real valued function $f\in L_{\rm loc}^{1}(X)$ and $B\subset X$, let $\alpha_B(f)$ be a real number such that

\begin{equation*}
    \inf_{c\in \mathbb{R}}\frac{1}{\mu(B)}\int_{B}|f(x)-c|d\mu(x)
\end{equation*}
is attained. Moreover, it is known that $\alpha_B(f)$ satisfies that

\begin{equation}\label{greater}
    \mu(\{x\in B: f(x)>\alpha_B(f)\})\leq \frac{\mu(B)}{2}
\end{equation}
and
\begin{equation}\label{lesser}
    \mu(\{x\in B: f(x)<\alpha_B(f)\})\leq \frac{\mu(B)}{2}.
\end{equation}
And it is easy to see that for any ball $B\subset X$,
\begin{align}\label{Msim}
M(b,B)\approx {1\over \mu(B)}\int_{B}\left|b(x)-\alpha_B(b)\right|d\mu(x),
\end{align}
where the implicit constants are independent of the function $b$ and the ball $B$.

By Lip$(\beta)$, $0<\beta<\infty$, we denote the set of all functions $\phi(x)$ defined on $X$ such that there exists a finite constant $C$ satisfying
$$|\phi(x)-\phi(y)|\leq Cd(x,y)^{\beta}$$
for every $x$ and $y$ in $X$. $\|\phi\|_{\beta}$ will stand for the least constant $C$ satisfying the condition above.
By ${\rm Lip}_{c}(\beta)$, we denote the set of all Lip$(\beta)$ functions with compact support on $X$.

\begin{defn}
We define ${\rm VMO}(X)$ as the closure of the ${\rm Lip}_{c}(\beta)$ functions $X$
 under the norm of the BMO space.
\end{defn}

We also need to establish the characterisation of ${\rm VMO}(X)$. We will give its proof in Appendix. For the Euclidean and the stratified Lie groups case one can refer to \cite{U1} and $\cite{Pxq}$.

\begin{lem}\label{lemvmo}
Let $f \in \mathrm{BMO}\left(X\right)$. Then $f \in
\mathrm{VMO}\left(X\right)$ if and only if $f$ satisfies the following three conditions:
\begin{enumerate}
\item[\rm(i)]$\lim\limits _{a \rightarrow 0} \sup\limits _{r_{B}=a} M(f, B)=0;$

\item[\rm(ii)]$\lim\limits _{a \rightarrow \infty} \sup\limits _{r_{B}=a} M(f, B)=0;$
\item[\rm (iii)] $\lim\limits _{r \rightarrow \infty} \sup\limits _{B \subset X \setminus B(x_0, r)} M(f, B)=0,$
\end{enumerate}
where $r_{B}$ is the radius of the ball $B$ and $x_0$ is a fixed point in $X$.
\end{lem}

To this end, we  recall the definition of $A_p$ weights.

\begin{defn}
  \label{def:Ap}
  Let $\omega(x)$ be a nonnegative locally integrable function
  on~$X$. For $1 < p < \infty$, we
  say $\omega$ is an $A_p$ \emph{weight}, written $\omega\in
  A_p$, if
  \[
    [\omega]_{A_p}
    := \sup_B \left(\intav_B \omega\right)
    \left(\intav_B
      \left(\dfrac{1}{\omega}\right)^{1/(p-1)}\right)^{p-1}
    < \infty.
  \]
  Here the suprema are taken over all balls~$B\subset X$.
  The quantity $[\omega]_{A_p}$ is called the \emph{$A_p$~constant
  of~$\omega$}.
For $p = 1$, we say $\omega$ is an $A_1$ \emph{weight},
  written $\omega\in A_1$, if $M(\omega)(x)\leq \omega(x)$ for $\mu$-almost every $x\in X$, and let $A_\infty := \cup_{1\leq p<\infty} A_p$ and we have
    \(
    [\omega]_{A_\infty}
    := \sup_B \left(\intav_B \omega\right)
    \exp\left(\intav_B \log \left(\frac{1}{\omega}\right) \right)
    < \infty.
  \)
\end{defn}

Next we note that for $\omega\in A_p$ the measure $\omega(x)d\mu(x)$ is a doubling measure on $X$. To be more precise, we have
that for all $\lambda>1$ and all balls $B\subset X$,
\begin{align}\label{doubling constant of weight}
\omega(\lambda B)\leq \lambda^{np}[\omega]_{A_p}\omega(B),
\end{align}
where $n$ is the upper dimension of the measure $\mu$, as in \eqref{upper dimension}.

We also point out that for $\omega\in A_\infty$, there exists $\gamma>0$ such that for every ball $B$,
$$ \mu\Big( \Big\{x\in B: \ \omega(x)\geq\gamma \intav_B \omega\Big\} \Big)\geq {1\over 2}\mu(B). $$
And this implies that for every ball $B$ and for all $\delta\in(0,1)$,
\begin{align}\label{e-reverse holder}
\intav_B \omega\le C \left(\intav_B \omega^\delta\right)^{1/\delta};
\end{align}
see  also \cite{LOR2}.

By the definition of $A_p$ weight and H\"older's inequality, we can easily obtain the  following
standard properties.
\begin{lem}\label{lemcomparison}
Let $\omega\in A_p(X)$, $p\geq 1$. Then there exists constants $\hat{C_1},\hat{C_2} >0$ and $\sigma\in(0,1)$ such that the following holds
\begin{equation*}
   \hat{ C_{1}}\left(\frac{\mu(E)}{\mu(B)}\right)^p \leq \frac{\omega(E)}{\omega(B)}\leq \hat{C_2}\left(\frac{\mu(E)}{\mu(B)}\right)^\sigma
\end{equation*}
for any measurable set $E$ of a quasi metric ball $B$.
\end{lem}

According to \cite[Theorem 5.5]{HT}, we have the following result for BMO functions on $X$.

\begin{lem}\label{bmoqq}
Let $0<p<\infty$, $v\in A_\infty(X)$, $f\in {\rm{BMO}}(X)$. Then
$$\|f\|_{{\rm{BMO}}(X)}\approx\sup_{B\subset X}\bigg\{{1\over v(B)}\int_B\big| f(x)-f_{B,v} \big|^pv(x)d\mu(x)  \bigg\}^{1\over p},$$
where $f_{B,v}={1\over v(B)}\int_{B}f(y)v(y)d\mu(y).$
\end{lem}

\section{Boundedness Characterization of Commutators}

In this section, we will give the proof of Theorem \ref{thm main1}.

\subsection{\textbf{Proof of Theorem \ref{thm main1}(i)}.}
In order to prove Theorem \ref{thm main1}(i), we need the following lemma.

\begin{lem}[{\cite{Dgk}}]\label{lem1}
Let $b\in BMO(X)$ and $T$ be Calder\'{o}n--Zygmund
operator on $(X, d,\mu)$ a Space of homogeneous type. If $\kappa\in(0,1), 1<p<\infty$ and $\omega\in A_{p}(X)$, then $[b, T]$ is bounded on $L_{\omega}^{p,\kappa}(X)$.

\end{lem}

\begin{proof}[ Proof of Theorem \ref{thm main1}(i)] Let $1<p<\infty$. Then it suffices to show that

\begin{equation*}\left(\frac{1}{[\omega(B)]^{\kappa}}\int_{B}|[b,T](x)|^{p}\omega(x)d\mu(x)\right)^{\frac{1}{p}}\lesssim \|b\|_{{\rm BMO}(X)}\|f\|_{L_{\omega}^{p,\kappa}(X)},
\end{equation*}

holds for any ball $B$.

Now  we will fix a ball $B= B(x_{0}, r)$ and then decompose $f= f\chi_{2B}+ f\chi_{X\setminus 2B} =: f_{1}+ f_{2} $. Then we have
\begin{align*}
&\frac{1}{\omega(B)^{\kappa}}\int_{B}|[b,T]f(x)|^{p}\omega(x)d\mu(x) \\
 &\lesssim
\left( \frac{1}{\omega(B)^{\kappa}}\int_{B}|[b,T]f_{1}(x)|^{p}\omega(x)d\mu(x) + \frac{1}{\omega(B)^{\kappa}}\int_{B}|[b,T]f_{2}(x)|^{p}\omega(x)d\mu(x)\right) \\
& =: I + II.
\end{align*}

For the first term $I$ here, we use Lemma \ref{lem1} and we obtain
\begin{align*}
\frac{1}{\omega(B)^{\kappa}}\int_{B}|[b,T]f_{1}(x)|^{p}\omega(x)d\mu(x) &\leq \frac{1}{\omega(B)^{\kappa}}\int_{X}|[b,T]f_{1}(x)|^{p}\omega(x)d\mu(x)\\
&\lesssim \|b\|_{{\rm BMO}(X)}^{p}\frac{1}{\omega(B)^{\kappa}}\int_{2B}|f(x)|^{p}\omega(x)d\mu(x)\\
&\lesssim \|b\|_{{\rm BMO}(X)}^{p}\|f\|_{L_{\omega}^{p,\kappa}(X)}^{p}.
\end{align*}
So we have
\begin{equation*}
    \|[b, T]f_{1}\|_{L_{\omega}^{p,\kappa}(X)}\lesssim \|b\|_{{\rm BMO}(X)}^{p}\|f\|_{L_{\omega}^{p,\kappa}(X)}^{p}.
\end{equation*}
Now for the second term $II$, observe that for $x\in B$, by \eqref {size of C-Z-S-I-O}, we have
\begin{align*}
   | [b, T]f_{2}(x)|^{p} &\leq \left(\int_{X}|b(x)-b(y)||K(x,y)||f_{2}(y)|d\mu(y)\right)^{p}\\
  &\lesssim \left(\int_{X\setminus2B} \frac{|b(x)-b(y)|}{V(x,y)}|f(y)|d\mu(y)\right)^{p}\\
   &\lesssim\left(\int_{X\setminus2B}\frac{|f(y)|}{V(x_0,y)}\{|b(x)-b_{B, \omega}|+|b_{B, \omega}-b(y)|\}d\mu(y)\right)^p\\
   &\lesssim \left(\int_{X\setminus2B}\frac{|f(y)|}{V(x_0,y)}d\mu(y)\right)^p  |b(x)-b_{B, \omega}|^p + \left(\int_{X\setminus2B}\frac{|f(y)|}{V(x_0,y)}|b_{B, \omega}-b(y)|d\mu(y)\right)^p,
\end{align*}
where $b_{B, \omega} = \frac{1}{\omega(B)}\int_{B}b(y)\omega(y)d\mu(y)$. Hence we have the following
\begin{align*}
    \frac{1}{\omega(B)^{\kappa}}\int_{B}|[b,T]f_{2}(x)|^{p}\omega(x)d\mu(x)
    &\lesssim \frac{1}{\omega(B)^{\kappa}} \left(\int_{X\setminus2B}\frac{|f(y)|}{V(x_0,y)}d\mu(y)\right)^p \int_{B} |b(x)-b_{B, \omega}|^p \omega(x)d\mu(x)\\ &\ \  + \left(\int_{X\setminus2B}\frac{|f(y)|}{V(x_0,y)}|b_{B, \omega}-b(y)|d\mu(y)\right)^p\omega(B)^{1-\kappa}\\
    &=: III + IV.
\end{align*}
Note that  $\lim\limits_{k\to \infty}\mu(2^kB)=\infty$. Then there exist $j_k\in \mathbb{N}$ such that
$$\mu(2^{j_1}B)\geq 2\mu(B)\  {\rm and }\ \mu(2^{j_{k+1}}B)\geq 2\mu(2^{j_k}B).$$
 For $III$, using the H\"older inequality, and using Lemma \ref{lemcomparison} and Lemma \ref{bmoqq}, we get

\begin{align*}
    III &\lesssim \|f\|_{L_{\omega}^{p,\kappa}(X)}^{p}\frac{1}{\omega(B)^\kappa}\left(\sum_{k=0}^{\infty}\int_{2^{j_{k+1}}B\setminus 2^{j_{k}}B}\frac{|f(y)|}{V(x_0,y)}d\mu(y)\right)^p
    \int_{B}|b(x)-b_{B,\omega}|^p\omega(x)d\mu(x)\\
     &\lesssim \|f\|_{L_{\omega}^{p,\kappa}(X)}^{p}\frac{1}{\omega(B)^\kappa}\left(\sum_{k=0}^{\infty}\frac{1}{\omega(2^{j_{k+1}}B)^\frac{1-\kappa}{p}}\right)^p
    \int_{B}|b(x)-b_{B,\omega}|^p\omega(x)d\mu(x)\\
    &\lesssim\|f\|_{L_{\omega}^{p,\kappa}(X)}^{p}\|b\|_{\rm BMO(X)}^p \left(\sum_{k=0}^{\infty}\left(\frac{\omega(B)}{\omega(2^{j_{k+1}}B)}\right)^\frac{1-\kappa}{p}\right)^p\\
     &\lesssim\|f\|_{L_{\omega}^{p,\kappa}(X)}^{p}\|b\|_{\rm BMO(X)}^p \left(\sum_{k=0}^{\infty}2^{-k\sigma\frac{1-\kappa}{p}}\right)^p\\
    &\lesssim \|f\|_{L_{\omega}^{p,\kappa}(X)}^{p}\|b\|_{\rm BMO(X)}^p.
\end{align*}
Using H\"older's inequality for the term $IV$, we get
\begin{equation*}
\begin{aligned}
IV &\lesssim  \left(\sum_{k=0}^{\infty}\frac{1}{\mu(2^{j_{k}}B)}\int_{2^{j_{k+1}}B}|f(y)||b_{B, \omega}-b(y)|d\mu(y)\right)^p\omega(B)^{1-\kappa}\\
&\lesssim \bigg(\sum_{k=0}^{\infty}\frac{1}{\mu(2^{j_{k }}B)}\left(\int_{2^{j_{k+1}}B}|f(y)|^p\omega(y)d\mu(y)\right)^\frac{1}{p}\\
&\ \ \times \left(\int_{2^{j_{k+1}}B}|b_{B, \omega}-b(y)|^{p'}\omega(y)^{1-p'}d\mu(y)\right)^\frac{1}{p'} \bigg)^p \omega(B)^{1-\kappa}\\
&\lesssim \|f\|_{L_{\omega}^{p,\kappa}(X)}^{p}\left\{\sum_{k=0}^{\infty}\frac{\omega(2^{j_{k+1}}B)^\frac{\kappa}{p}}{\mu(2^{j_{k }}B)}\left(\int_{2^{j_{k+1}}B}|b_{B, \omega}-b(y)|^{p'}\omega(y)^{1-p'}d\mu(y)\right)^\frac{1}{p'}\right\}^p\omega(B)^{1-\kappa}.
\end{aligned}
\end{equation*}
Now observe that
\begin{align*}
&\left(\int_{2^{j_{k+1}}B}|b_{B, \omega}-b(y)|^{p'}\omega(y)^{1-p'}d\mu(y)\right)^\frac{1}{p'}\\ &\leq\left(\int_{2^{j_{k+1}}B}\left(|b(y)-b_{2^{j_{k+1}}B,\omega^{1-p'}}|+|b_{2^{j_{k+1}}B,\omega^{1-p'}}-b_{B,\omega}|\right)^{p'}\omega(y)^{1-p'}d\mu(y)\right)^\frac{1}{p'}\\
&\leq \left(\int_{2^{j_{k+1}}B}\left(|b(y)-b_{2^{j_{k+1}}B,\omega^{1-p'}}|\right)^{p'}\omega(y)^{1-p'}d\mu(y)\right)^\frac{1}{p'} \\\ & \hspace{1 cm}+ \left(\int_{2^{j_{k+1}}B}\left(|b_{2^{j_{k+1}}B,\omega^{1-p'}}-b_{B,\omega}|\right)^{p'}\omega(y)^{1-p'}d\mu(y)\right)^\frac{1}{p'}\\
&\hspace{0.5 cm}=: V + VI.
\end{align*}
We have $\omega^{1-p'}\in A_{p'}(X)$ since $\omega\in A_p(X)$. So we obtain
\begin{equation*}
    V \lesssim \|b\|_{\rm BMO(X)}\omega^{1-p'}(2^{j_{k+1}}B)^\frac{1}{p'}.
\end{equation*}
For $VI$, we have
\begin{align*}
  \left | b_{2^{j_{k+1}}B,\omega^{1-p'}} - b_{B,\omega}\right | & \leq \left | b_{2^{j_{k+1}}B,\omega^{1-p'}} - b_{2^{j_{k+1}}B} \right | + \left | b_{2^{j_{k+1}}B} - b_{B} \right | + \left | b_{B} - b_{B,\omega} \right | \\
 &\lesssim \frac{1}{\omega^{1-p'(2^{j_{k+1}}B)}}\int_{2^{j_{k+1}}B}\left | b(y)-b_{2^{j_{k+1}}B} \right |\omega(y)^{1-p'}d\mu(y) \\
   &+  (k+1)\left \| b \right \|_{\rm BMO(X)}+\frac{1}{\omega(B)}\int_{B}\left | b(y)-b_{B} \right |\omega(y)d\mu(y).
\end{align*}
As we have $b\in {\rm {BMO}}(X)$, by  Lemma  \ref{lem-jn1}, there exists some constants $C_1>0$ and $C_2 >0$ such that for any  ball B and $\alpha >0$
\begin{equation*}
\mu(\{x\in B: | b(x)-b_{B}|> \alpha \} )\leq C_{1} \mu( B )e^{-\frac{C_{2}\alpha }{\| b \|_{\rm BMO(X)}}}.
\end{equation*}
Then using Lemma \ref{lemcomparison}, we get

\begin{equation*}
    \omega(\{x\in B:\left | b(x)-b_{B}\right |> \alpha  \})\leq C_{1}\omega(B)e^{-\frac{C_{2}\alpha\sigma }{\left \| b \right \|_{\rm BMO(X)}}}
\end{equation*}
for some $\sigma\in(0,1)$. Hence we have

\begin{align*}
    \int_{B}\left | b(y)-b_{B} \right |\omega(y)d\mu(y)&=\int_{0}^{\infty}\omega(\{y\in B:\left | b(y)-b_{B} \right |>\alpha\})d\alpha \\
    &\lesssim \omega(B)\int_{0}^{\infty}e^{-\frac{\bar{C_{2}}\alpha\sigma }{\left \| b \right \|_{\rm BMO(X)}}}d\alpha \\
    &\lesssim \omega(B)\left \| b \right \|_{\rm BMO(X)}.
    \end{align*}
Similarly, we also get
\begin{equation*}
    \Big(\int_{2^{j_{k+1}}B}\left | b(y)-b_{2^{j_{k+1}}B} \right |\omega(y)^{1-p'}d\mu(y)\Big)^{\frac{1}{p'} }\lesssim(k+1)\left \| b \right \|_{\rm BMO(X
    )}\omega^{1-p'}(2^{j_{k+1}}B)^{1/p'}.
\end{equation*}
Together with Lemma \ref{lemcomparison}, we have the following
\begin{align*}
    IV&\lesssim\left \| f \right \|_{L_{\omega}^{p,\kappa}(X)}^{p}\left \| b \right \|_{\rm BMO(X)}^{p}\left[\sum_{k=0}^{\infty}\frac{\omega(2^{j_{k+1}}B)^{\frac{\kappa}{p}}}{\mu( 2^{j_{k }}B )}(k+1)\omega^{1-p'}(2^{j_{k+1}}B)^{1/p'}\right]^p\omega(B)^{1-\kappa}\\
    &\lesssim\left \| f \right \|_{L_{\omega}^{p,\kappa}(X)}^{p}\left \| b \right \|_{\rm BMO(X)}^{p}\left[\sum_{k=0}^{\infty}\frac{(k+1)\omega(B)^{\frac{1-\kappa}{p}}}{\omega(2^{j_{k+1}}B)^{\frac{1-\kappa}{p}}}\right]^p\\
    &\lesssim\left \| f \right \|_{L_{\omega}^{p,\kappa}(X)}^{p}\left \| b \right \|_{\rm BMO(X)}^{p}\left[\sum_{k=0}^{\infty}(k+1)2^{-\frac{(k+1)(1-\kappa) \sigma}{p}}\right]^p\\
    &\lesssim\left \| f \right \|_{L_{\omega}^{p,\kappa}(X)}^{p}\left \| b \right \|_{\rm BMO(X)}^{p}.
\end{align*}
Therefore we have
\begin{equation*}
    \left \| [b,T]f_{2} \right \|_{L_{\omega}^{p,\kappa}(X)}\lesssim\left \| f \right \|_{L_{\omega}^{p,\kappa}(X)}\left \| b \right \|_{\rm BMO(X)}.
\end{equation*}
This completes the proof.
\end{proof}
\vspace{0.5 cm}

\subsection{\textbf{Proof of Theorem \ref{thm main1}(ii)}.}

We first recall another version of the homogeneous condition (formulated in \cite{Dgk}): there exist positive constants $3\le A_1\le A_2$ such that for any ball $B:=B(x_0, r)\subset X$, there exist balls $\widetilde B:=B(y_0, r)$
such that $A_1 r\le d(x_0, y_0)\le A_2 r$, 
and  for all $(x,y)\in ( B\times \widetilde{B})$, $K(x, y)$ does not change sign and
\begin{equation}\label{e-assump cz ker low bdd}
|K(x, y)|\gs \frac1{\mu(B)}.
\end{equation}
If the kernel $K(x, y):=K_1(x,y)+i K_2(x,y)$ is complex-valued, where $i^2=-1$,
then at least one of $K_i$ satisfies \eqref{e-assump cz ker low bdd}.

Then we first point out that the homogeneous condition \eqref{lowerbound} implies
\eqref{e-assump cz ker low bdd}.

\begin{lem}[\cite{Dgk}]\label{prop homogeneous}
Let $T$ be the Calder\'on--Zygmund operator as in Definition \ref{def 1} and satisfy the homogeneous condition as in
\eqref{lowerbound}. Then $T$ satisfies \eqref{e-assump cz ker low bdd}.
\end{lem}

\begin{proof}[Proof of Theorem \ref{thm main1}(ii)]
To prove $b\in BMO(X)$, it is sufficient to show for any ball $B\subset X$, we have $M(b,B)\lesssim 1$. Let $B = B(x_0,r)$ be a quasi metric ball in $X$. Also let $\widetilde{B} := B(y_0,r)\subset X$ be the measurable set in (\ref{e-assump cz ker low bdd}).
Following \cite{Dgk}, we take
\begin{equation*}
    E_1 := \{x\in B: b(x)\geq \alpha_{\tilde{B}}(b)\} \quad
    E_2 := \{x\in B: b(x) <  \alpha_{\tilde{B}}(b)\};
\end{equation*}
\begin{equation*}
    F_1 \subset \{y\in \tilde{B}: b(y)\leq \alpha_{\tilde{B}}(b)\} \quad
    F_2 \subset \{y\in \tilde{B}: b(y) \geq \alpha_{\tilde{B}}(b)\},
\end{equation*}
with $\alpha_{\tilde{B}}(b)$ the median value of $b$ over $\tilde B$, such that $\mu(F_1)= \mu(F_2)= \frac{1}{2}\mu(\tilde{B})$ and $F_{1}\cap F_{2} = \emptyset$. For any  $(x,y)\in E_j \times F_j$, $j\in\{1,2\}$, we have

\begin{equation*}
    |b(x)-b(y)| = |b(x) -  \alpha_{\tilde{B}}(b)| + | \alpha_{\tilde{B}}(b)-b(y)| \geq |b(x)- \alpha_{\tilde{B}}(b)|.
\end{equation*}

As b is a real valued, using Lemma \ref{lemcomparison},   H\"older's inequality, and using the boundedness of $[b,T]$ on $L_{\omega}^{p,\kappa}(X)$ and (\ref{e-assump cz ker low bdd}), we get that

\begin{align*}
    M(b,B)&\lesssim\frac{1}{\mu( B )}\int_{B}\left | b(x)-\alpha_{\tilde{B}}(b) \right |d\mu(x)\approx \sum_{j=1}^{2}\frac{1}{\mu(B)}\int_{E_j}|b(x)-\alpha_{\tilde{B}}(b)|d\mu(x)\\
    &\lesssim\sum_{j=1}^{2}\frac{1}{\mu( B)}\int_{E_j}\int_{F_j}\frac{\left | b(x)-\alpha_{\tilde B}(b) \right |}{\mu( B )}d\mu(y)d\mu(x)\\
    &\approx \sum_{j=1}^{2}\frac{1}{\mu( B)}\int_{E_j}\int_{F_j}\frac{\left | b(x)-\alpha_{\tilde B}(b) \right |}{V(x,y)}d\mu(y)d\mu(x)\\
    &\lesssim \sum_{j=1}^{2}\frac{1}{\mu( B )}\int_{E_j}\int_{F_j}\left | b(x)-b(y) \right |\frac{1}{V(x,y)}d\mu(y)d\mu(x)\\
    &\lesssim \sum_{j=1}^{2}\frac{1}{\mu( B )}\int_{E_j}\left | \int_{F_j}\left | b(x)-b(y) \right | K(x,y)d\mu(y)\right |d\mu(x)\\
    &\sim \sum_{j=1}^{2}\frac{1}{\mu( B)}\int_{E_j}\left | [b,T]\chi_{ F_{j}(x)} \right |d\mu(x)\\
    &\lesssim  \sum_{j=1}^{2}\frac{1}{\mu( B )}\int_{E_j}\left | [b,T]\chi_{ F_{j}(x)} \right |d\mu(x)\lesssim\sum_{j=1}^{2}\frac{1}{\mu( B)}\int_{E_j}\left \|  [b,T]\chi_{ F_{j}}\right \|_{L_{\omega}^{p,\kappa}(X)}[\omega(B)]^{\frac{\kappa-1}{p}}\mu(B)\\
    &\lesssim\sum_{j=1}^{2}\left \|  [b,T]\right \|_{L_{\omega}^{p,\kappa}(X)\rightarrow L_{\omega}^{p,\kappa}(X)}\left \| \chi_{ F_{j}} \right \|_{L_{\omega}^{p,\kappa}(X)}[\omega(B)]^\frac{\kappa-1}{p}\\
    &\lesssim\left \|  [b,T]\right \|_{L_{\omega}^{p,\kappa}(X)\rightarrow L_{\omega}^{p,\kappa}(X)}[\omega(\tilde{B})]^\frac{1-\kappa}{p}[\omega(B)]^\frac{\kappa-1}{p}\\
    &\lesssim\left \| [b,T]\right \|_{L_{\omega}^{p,\kappa}(X)\rightarrow L_{\omega}^{p,\kappa}(X)}.
\end{align*}
This completes the proof of Theorem \ref{thm main1}(ii).
\end{proof}

\section{Compactness Characterization of the Commutator}

Now we will prove Theorem \ref{thm main2}.

\subsection{\textbf{Proof of Theorem \ref{thm main2}(i)}.}

Now we will give sufficient conditions for the subsets of weighted Morrey spaces to be relatively compact. We define a subset $\mathcal{F}$ of $L_{\omega}^{p,\kappa}(X)$ to be totally bounded if the $L_{\omega}^{p,\kappa}(X)$ closure of $\mathcal{F}$ is compact.

\begin{lem}\label{lemrelcompact}
Let $p\in(1,\infty), \kappa\in (0,1)$ and $\omega \in A_p(X)$, then a subset $\mathcal{F}$ of $L_{\omega}^{p,\kappa}(X)$ is totally bounded if the set $\mathcal{F}$ satisfies the following three conditions:
\begin{enumerate}
    \item[(i)] $\mathcal{F}$ is bounded, namely,
    \begin{equation*}
        \sup\limits_{f\in\mathcal{F}}\|f\|_{L_{\omega}^{p,\kappa}(X)} < \infty;
    \end{equation*}
    \item[(ii)] $\mathcal{F}$ vanishes uniformly at infinity, namely, for any $\epsilon\in(0,\infty)$, there exists some positive constant $M$ such that, for any $f\in \mathcal{F}$,
    \begin{equation*}
        \|f\chi_{\{x\in X:d(x_0,x)>M\}}\|_{L_{\omega}^{p,\kappa}(X)}<\epsilon,
    \end{equation*}
    where $x_0$ is a fixed point in $X$;
    \item[(iii)] $\mathcal{F}$ is uniformly equicontinuous, namely,
    \begin{equation*}
      \lim_{r\rightarrow 0}\|f(x)-f_{B(x,r)}\|_{L_{\omega}^{p,\kappa}(X)} =0
    \end{equation*}
    uniformly for $f\in\mathcal F$.
    \end{enumerate}
\end{lem}

The proof for the lemma above, follows from \cite{MaoSunWu16ActaMathSin} using a small modification from Euclidean setting to space of homogeneous type, this only requires following properties of the underlying space: metric on space and doubling measure.

We will now show the boundedness of the maximal operator $T_*$ of a family of smooth truncated operators $\{T_{\eta}\}_{\eta\in(0,\infty)}$ as follows. For $\eta\in(0,\infty)$, we take

\begin{equation*}
    T_{\eta}f(x) := \int_{X}K_{\eta}(x,y)f(y)d\mu(y),
\end{equation*}
where the kernel $K_{\eta} := K(x,y)\varphi(\frac{d(x,y)}{\eta})$ with $\varphi\in C^{\infty}(\mathbf{R})$ and $\varphi$ satisfies the following
\[
\varphi(t)=\left\{\begin{array}{ll}

\varphi(t) \equiv 0 & \text { if } t \in\left(-\infty, \frac{1}{2}\right) \\
\varphi(t) \in[0,1], & \text { if } t \in\left[\frac{1}{2}, 1\right] \\
\varphi(t) \equiv 1, & \text { if } t \in(1, \infty).
\end{array}\right.
\]
Let
\[
\left[b, T_{\eta}\right] f(x):=\int_{X}[b(x)-b(y)] K_{\eta}(x, y) f(y) d\mu(y).
\]
The maximal operator $T_*$ is defined as below
\[
T_{*} f(x):=\sup _{\eta \in(0, \infty)}\left|\int_{X} K_{\eta}(x, y) f(y) d\mu(y)\right|.
\]
Observe that the $\emph{Hardy-Littlewood maximal Operator}$ $\mathcal{M}$ is given as
\[
\mathcal{M} f(x):=\sup _{B \ni x} \frac{1}{\mu(B)} \int_{B}|f(y)| d\mu(y)
\]
for any $f \in L_{\mathrm{loc}}^{1}\left(X\right)$ and $x \in X$, here we take the supremum over all quasi-metric  balls $B$ of $X$ that contain $x$.



Then we have the following lemmas.
\begin{lem}\label{lembound}
There exists a positive constant $C$ such that we have, for any $b\in $Lip$(\beta)$, $0<\beta<\infty$, $f \in
L_{\operatorname{loc}}^{1}\left(X\right)$ and $x \in X$
$$
\left|\left[b, T_{\eta}\right] f(x)-[b, T] f(x)\right| \leq C\eta^{\beta} \mathcal{M} f(x).
$$
\end{lem}
\begin{proof}
Let $f \in L_{\text {loc}}^{1}\left(X\right)$. Now for any $x \in X$, we get
\begin{align*}
    &| [b, T_{\eta}]f(x)-[b, T] f(x)| \\
&\quad =\left|\int_{\eta / 2<d(x, y) \leq \eta}[b(x)-b(y)] K_{\eta}(x, y) f(y) d\mu(y)-\int_{d(x, y) \leq \eta}[b(x)-
b(y)] K(x, y) f(y) d\mu(y)\right| \\
&\quad \lesssim \int_{d(x, y) \leq \eta}|b(x)-b(y)||K(x, y)||f(y)| d\mu(y).
\end{align*}
Since $b\in $Lip$(\beta)$ and (\ref{size of C-Z-S-I-O}), we have that
\begin{align*}
   &\int_{d(x, y) \leq \eta}|b(x)-b(y)||K(x, y)||f(y)| d\mu(y)\\
&\quad \lesssim C\sum_{j=0}^{\infty}\int_{\eta2^{-(j+1)}<d(x,y)\leq \eta2^{-j}}\frac{d(x,y)^{\beta}}{V(x,y)}|f(y)|d\mu(y)\\
&\quad \lesssim C\eta^{\beta} \mathcal{M}f(x),
\end{align*}
this completes the proof of the Lemma \ref{lembound}.
\end{proof}
\begin{lem}\label{lemmaximal}
Let $p \in(1, \infty), \kappa \in(0,1)$ and $\omega \in A_{p}\left(X\right) .$ Then there exists a positive constant $C$ such that, for any $f \in L_{\omega}^{p, \kappa}\left(X\right)$,
\begin{equation*}
    \|T_{*}\|_{ L_{\omega}^{p, \kappa}\left(X\right)}+ \|\mathcal{M}f\|_{ L_{\omega}^{p, \kappa}\left(X\right)}\leq C\|f\|_{ L_{\omega}^{p, \kappa}\left(X\right)}.
\end{equation*}
\end{lem}
\begin{proof} To show the boundedness of $\mathcal{M}$ on $L_{\omega}^{p, \kappa}\left(X\right)$ one can refer to \cite{HT}. We will now only consider the boundedness of $ {T}_{*}$. For any fixed quasi-metric ball $B \subset X$ and $f \in L_{\omega}^{p, \kappa}\left(X\right),$ we write the following
\[
f:=f_{1}+f_{2}:=f \chi_{2 B}+f \chi_{X \backslash 2 B}.
\]
Note that  $\lim\limits_{k\to \infty}\mu(2^kB)=\infty$. Then there exist $j_k\in \mathbb{N}$ such that
$$\mu(2^{j_1}B)\geq 2\mu(B)\  {\rm and }\ \mu(2^{j_{k+1}}B)\geq 2\mu(2^{j_k}B).$$ 
Observe that $f_{1} \in L_{\omega}^{p}\left(X\right) .$ Then using the boundedness of $ {T}_{*}$ on $L_{\omega}^{p}\left(X\right)$ (see, for example, \cite[Theorem 1.1]{Gxq} ) and from the H\"{o}lder inequality, also using size and smoothness of Kernel, we have that
\begin{align*}
   & {\left[\int_{B}\left|T_{*} f(x)\right|^{p} \omega(x) d\mu(x)\right]^{\frac{1}{p}}} \\
&\quad \lesssim\left[\int_{B}\left|T_{*} f_{1}(x)\right|^{p} \omega(x) d\mu(x)\right]^{\frac{1}{p}}+\sum_{k=0}^{\infty}\left\{\int_{B}\left[\int_{2^{j_{k+1}} B \backslash 2^{j_{k }} B} \frac{|f(y)|}{V(x,y)} d\mu(y)\right]^{p} \omega(x) d\mu(x)\right\}^{\frac{1}{p}} \\
&\quad \lesssim\left[\int_{2 B}|f(x)|^{p} \omega(x) d\mu(x)\right]^{\frac{1}{p}}+\sum_{k=0}^{\infty}\left[\frac{\omega(B)}{\mu(2^{j_{k }} B)^{p}}\left\{\int_{2^{j_{k+1}} B}|f(y)|[\omega(y)]^{\frac{1}{p}}[\omega(y)]^{-\frac{1}{p}} d\mu( y)\right\}^{p}\right]^{\frac{1}{p}}\\
&\quad \lesssim\|f\|_{L_{\omega}^{p, \kappa}\left(X\right)}[\omega(B)]^{\frac{\kappa}{p}}+\sum_{k=0}^{\infty}\left\{\omega(B)\left[\omega\left(2^{j_{k }} B\right)\right]^{\kappa-1}\|f\|_{L_{\omega}^{p, \kappa}\left(X\right)}^{p}\right\}^{\frac{1}{p}}\\
&\quad \lesssim\|f\|_{L_{\omega}^{p, \kappa}\left(X\right)} \omega(B)^{\frac{\kappa}{p}}+\sum_{k=0}^{\infty}\left\{[\omega(B)]^{\kappa} 2^{-k \sigma(1-\kappa )}\|f\|_{L_{\omega}^{p, \kappa}\left(X\right)}^{p}\right\}^{\frac{1}{p}}\\
&\quad\lesssim\|f\|_{L_{\omega}^{p, \kappa}\left(X\right)}[\omega(B)]^{\frac{\kappa}{p}},
\end{align*}
in the fourth inequality above, we have used Lemma \ref{lemcomparison} for some $\sigma\in(0,1)$. This completes the proof of Lemma \ref{lemmaximal}.
\end{proof}

\begin{proof}[Proof of Theorem $\ref{thm main2}(i)$ ] When $b \in \mathrm{VMO}\left(X\right),$ then for any $\varepsilon \in(0, \infty),$ there exists $b^{(\varepsilon)} \in$ ${\rm Lip}_{c}(\beta), 0<\beta<\infty$ such that we have $\left\|b-b^{(\varepsilon)}\right\|_{\mathrm{BMO}\left(X\right)}<\varepsilon$. Then, using the boundedness of the commutator $[b,T]$ on $L_{\omega}^{p, \kappa}\left(X\right)$, we obtain

\[
\begin{aligned}
\left\|[b, T] f- [b^{(\varepsilon)}, T] f\right\|_{L_{\omega}^{p, \kappa}\left(X\right)} &=\left\|[b-b^{(\varepsilon)}, T] f\right\|_{L_{\omega}^{p, \kappa}\left(X\right)} \\
& \lesssim\left\|b-b^{(\varepsilon)}\right\|_{\mathrm{BMO}\left(X\right)}\|f\|_{L_{\omega}^{p, \kappa}\left(X\right)} \\
&\lesssim\varepsilon\|f\|_{L_{\omega}^{p, \kappa}\left(X\right)}.
\end{aligned}
\]
 Also using Lemmas \ref{lembound}  and \ref{lemmaximal}, we have the following

$$
\lim _{\eta \rightarrow 0}\left\|\left[b, T_{\eta}\right]-[b,T]\right\|_{L_{\omega}^{p, \kappa}\left(X\right) \rightarrow L_{\omega}^{p, \kappa}\left(X\right)}=0.
$$
It sufficient to show that, for any $b \in  $${\rm Lip}_{c}(\beta), 0<\beta<\infty$ and $\eta \in(0, \infty)$ small enough, $\left[b, T_{\eta}\right]$ is a compact operator on $L_{\omega}^{p, \kappa}\left(X\right),$ this is equivalent to showing that, for any bounded subset $\mathcal{F} \subset L_{\omega}^{p, \kappa}\left(X\right),\left[b, T_{\eta}\right] \mathcal{F}$ is relatively compact. Which means, we need to show that $\left[b, T_{\eta}\right]$
satisfies the conditions (i) through (iii) of Lemma \ref{lemrelcompact}.

Observe by \cite[Theorem  3.4]{KS} and using the fact that  $b \in \mathrm{BMO}\left(X\right)$, we have that $\left[b, T_{\eta}\right]$ is bounded on $L_{\omega}^{p, \kappa}\left(X\right)$ for the given $p \in(1, \infty), \kappa \in(0,1)$ and $\omega \in A_{p}\left(X\right)$, this shows that $\left[b, T_{\eta}\right] \mathcal{F}$ satisfies condition (i) of Lemma \ref{lemrelcompact}.

Now, let $x_0$ be a fixed point in $X$. Since $b \in  $${\rm Lip}_{c}(\beta),  $ we can further assume that $\|b\|_{L^{\infty}}= 1$. Recall that there exists a positive constant $R_{0}$ such that supp $(b) \subset B\left(x_0, R_{0}\right)$. Let $M \in\left(10 R_{0}, \infty\right) .$ Thus, for any $y \in B\left(x_0, R_{0}\right)$ and $x \in X$ with $d(x_0, x)>M,$
$d(x, y) \sim d(x_0, x)$. Then, for $x \in X$ with $d(x_0, x)>M,$ by H\"{o}lder inequality  and using that $V(x,y)\sim \mu(B(x,d(x,y)))$ we deduce that
\begin{align*}
  \left|\left[b, T_{\eta}\right] f(x)\right| & \leq \int_{X}|b(x)-b(y)|\left|K_{\eta}(x, y)\right||f(y)| d\mu(y) \\
 & \leq \int_{X}|b(y)|\left|K(x, y)\right||f(y)| d\mu(y) \\
& \lesssim  \int_{B\left(x_0, R_{0}\right)} \frac{|f(y)|}{V(x,y)} d\mu(y) \\
& \lesssim \int_{B\left(x_0, R_{0}\right)} \frac{|f(y)|}{\mu(B(x_0,d(x,x_0))} d\mu(y) \\
& \lesssim \frac{1}{\mu(B(x_0,d(x,x_0))} \left[\int_{B\left(x_0, R_{0}\right)}|f(y)|^{p} \omega(y) d\mu(y) \right]^{\frac{1}{p}}\left\{\int_{B\left(x_0, R_{0}\right)}[\omega(y)]^{-\frac{p^{\prime}}{p}} d\mu(y) \right\}^{\frac{1}{p'}} \\
& \lesssim \frac{\mu(B\left(x_0, R_{0}\right))}{\mu(B(x_0,d(x,x_0))}\left[\omega\left(B\left(x_0, R_{0}\right)\right)\right]^{\frac{\kappa-1}{p}} \|f\|_{L_{\omega}^{p, \kappa}\left(X\right)}.
\end{align*}

From $\lim\limits_{k\to \infty}\mu(B(x_0,kM))=\infty$, we have that there exist $j_k\in \mathbb{N}$ such that
$$\mu(B(x_0,2^{ j_1}M))\geq 2\mu(B(x_0,M))\  {\rm and }\ \mu(B(x_0,2^{ j_{k+1}}M))\geq 2\mu(B(x_0,2^{ j_{k}}M)).$$
Hence, for any fixed ball $B:=B(\widetilde{x}, \widetilde{r}) \subset X$, by Lemma \ref{lemcomparison} , we get that
\begin{align*}
&\frac{1}{[\omega(B)]^\kappa}\int_{B\cap \{x\in X:d(x_0, x)>M\}}|[b,T_{\eta}]f(x)|^p \omega(x)d\mu(x)\\
&\lesssim \mu(B\left(x_0, R_{0}\right))^p {\left[\omega\left(B\left(x_0, R_{0}\right)\right)\right]^{ (\kappa-1) }\over [\omega(B)]^\kappa }\|f\|_{L_{\omega}^{p, \kappa}\left(X\right)}  \sum_{k=0}^{\infty} \frac{\omega\left(B \cap\left\{x \in X: 2^{ j_{k }} M<d(x_0, x) \leq 2^{ j_{k+1}} M\right\}\right)}{\mu(B(x_0,2^{ j_{k }} M))^p  }\\
&\lesssim  \|f\|_{L_{\omega}^{p, \kappa}\left(X\right)}  \sum_{k=0}^{\infty} \frac{\omega\left(B (x_0, 2^{ j_{k+1}} M)\right)^{1-\kappa}}{\omega\left(B\left(x_0, R_{0}\right)\right)^{1-\kappa} } \frac{\mu(B\left(x_0, R_{0}\right))^p  }{\mu(B(x_0,2^{ j_{k }} M))^p  }\\
&\lesssim \|f\|_{L_{\omega}^{p,\kappa}}  \sum_{k=0}^{\infty} \frac{\mu(B\left(x_0, R_{0}\right))^{p\kappa}  }{\mu(B(x_0,2^{ j_{k }} M))^{p\kappa} }\\
&\lesssim \|f\|_{L_{\omega}^{p,\kappa}} \sum_{k=0}^{\infty}2^{- k}  \frac{\mu(B\left(x_0, R_{0}\right))^{p\kappa}  }{\mu(B(x_0, M))^{p\kappa} }\\
&\lesssim  \frac{\mu(B\left(x_0, R_{0}\right))^{p\kappa}  }{\mu(B(x_0, M))^{p\kappa} }\|f\|_{L_{\omega}^{p,\kappa}\left(X\right)}^p.
\end{align*}
Therefore the condition $(ii)$ of Lemma \ref{lemrelcompact} holds for $[b,T_{\eta}]\mathcal{F}$ with large $M$.

Now we will prove  $[b,T_{\eta}]\mathcal{F}$ also satisfies $(iii)$ of Lemma \ref{lemrelcompact}.
Let $\eta$ be a fixed positive constant small enough and $r<{\eta \over 8A_0^2}$. Now, for any $x \in X$, we have
\begin{align*}
[b, T_{\eta}]f(x)-\left([b,T_{\eta}]f \right)_{B(x,r)}
&={1\over \mu(B(x,r))}\int_{B(x,r)} [b, T_{\eta}]f(x)-[b, T_{\eta}]f(y) d\mu(y).
\end{align*}
Note that
\begin{align*}
&{\left[b, T_{\eta}\right] f(x)-\left[b, T_{\eta}\right] f(y)} \\
&\quad=[b(x)-b(y] \int_{X} K_{\eta}(x, z) f(z) d\mu(z) +\int_{X}\left[K_{\eta}(x, z)-K_{\eta}(y, z)\right][b(y)-b(z)] f(z) d\mu(z)\\
&\quad =: L_{1}(x,y)+L_{2}(x,y).
\end{align*}
As $b \in  $${\rm Lip}_{c}(\beta),  $  it follows that, for any $y \in B(x,r)$
\[
\left|L_{1}(x,y)\right|=|b(x)-b(y)|\left|\int_{X} K_{\eta}(x, z) f(z) d\mu(z)\right| \lesssim r^{\beta} {T}_{*}(f)(x).
\]

To estimate $L_{2}(x,y),$ we first recall that $K_{\eta}(x, z)=0, K_{\eta}(y, z)=0$ for
any $y \in B(x,r)$,  $d(x, z) \leq {\eta \over 4A_0}$ and $r<{\eta \over 8A_0^2}$.
Using the definition of $K_{\eta} $ we have that, for any $y \in B(x,r)$, $d(x, z) > {\eta \over 4A_0}$ and $r<{\eta \over 8A_0^2}$,

\[
\left|K_{\eta}(x, z)-K_{\eta}(y, z)\right| \lesssim \frac{1 }{V(x,z)}{d(x,y)^{\sigma_0}\over d(x,z)^{\sigma_0}}.
\]
Hence this implies, for any $y \in B(x,r)$
\[
\begin{aligned}
\left|L_{2}(x,y)\right| & \lesssim    \int_{d(x, z)>{\eta \over 4A_0}} \frac{|f(z)|}
{V(x, z)} {d(x,y)^{\sigma_0}\over d(x,z)^{\sigma_0}}d\mu(z) \\
& \lesssim \sum_{k=0}^{\infty} \frac{r^{\sigma_0}}{\left(2^{k} \eta\right)^{\sigma_0}}{1\over \mu(B(x,{2^k\eta \over 4A_0}))} \int_{{2^k\eta \over 4A_0}<d(x, z) \leq {2^{k+1}\eta \over 4A_0}}|f(z)| d\mu(z) \\
& \lesssim   \frac{r^{\sigma_0}}{\eta^{\sigma_0}}\mathcal{M} f(x).
\end{aligned}
\]
Using the estimates of  $L_1(x,y)$ and $L_2(x,y)$, we have
$$\Big|[b, T_{\eta}]f(x)-\left([b,T_{\eta}]f \right)_{B(x,r)}\Big|\lesssim  r^{\beta} {T}_{*}(f)(x)+\frac{r^{\sigma_0}}{\eta^{\sigma_0}}\mathcal{M} f(x).$$
Then, using Lemma \ref{lemmaximal} and  the boundedness of $\mathcal{M}$ on $L_{\omega}^{p,\kappa}\left(X\right),$ we obtain

\begin{equation*}
    \|[b, T_{\eta}]f(x)-\left([b,T_{\eta}]f \right)_{B(x,r)}\|_{L_{\omega}^{p,\kappa}}\lesssim  (r^{\beta}+\frac{r^{\sigma_0}}{\eta^{\sigma_0}})\|f\|_{L_{\omega}^{p,\kappa}}.
\end{equation*}
Hence we observe that, $\left[b, T_{n}\right] \mathcal{F}$ satisfies condition $(iii)$ of Lemma \ref{lemrelcompact}. 
So we have that,
$\left[b, T_{n}\right]$ is a compact operator for any $b \in  $${\rm Lip}_{c}(\beta) $. This completes the proof of Theorem \ref{thm main2}(i).
\end{proof}

\subsection{\textbf{Proof of Theorem \ref{thm main2}(ii)}.}


Next, we establish a lemma for the upper and the lower bounds of integrals of $[b, T] f_{j}$ on
certain balls $B_{j}$ in $X$ for any $j \in \mathbb{N}$.

\begin{lem}\label{lemsequence}
Let $p \in(1, \infty), \kappa \in(0,1)$ and $\omega \in A_{p}\left(X\right)$. Suppose that $b \in \mathrm{BMO}\left(X\right)$ is a real-valued function with $\|b\|_{\mathrm{BMO}\left(X\right)}=1$ and there exists
$\gamma \in(0, \infty)$ and a sequence $\left\{B_{j}\right\}_{j \in \mathrm{N}}:=\left\{B\left(x_{j},
r_{j}\right)\right\}_{j \in \mathrm{N}}$ of balls in $X$, with $\left\{x_{j}\right\}_{j \in
\mathrm{N}} \subset X$ and $\left\{r_{j}\right\}_{j \in \mathrm{N}} \subset(0,\infty)$ such that, for any $j \in \mathbb{N}$
\begin{equation}\label{mdel}
M\left(b, B_{j}\right)>\gamma.
\end{equation}
Then there exist real-valued functions $\left\{f_{j}\right\}_{j \in \mathrm{N}} \subset L_{\omega}^{p,\kappa}\left(X\right),$ positive constants $K_{0}$ large enough, $\widetilde{C}_{0}, \widetilde{C}_{1}$ and $\widetilde{C}_{2}$ such that, for any $j \in
\mathbb{N}$ and integer $k \geq K_{0},\left\|f_{j}\right\|_{L_{\omega}^{p,\kappa}\left(X\right)} \leq \bar{C}_{0}$,

\begin{equation}\label{communorm}
\int_{B_{j}^{k}} |\left[b, T] f_{j}(x)\right|^{p} \omega(x) d\mu(x) \geq \tilde{C}_{1}  \frac{  \gamma^{p}\mu( B_j)^p}{\mu(A_2^kB_j)^p} \left[\omega\left(B_{j}\right)\right]^{\kappa-1} \omega\left(A_2^{k} B_{j}\right),
\end{equation}
where $B_{j}^{k}:= \widetilde{A_2^{k-1} B_{j}}$ is the ball associates with $A_2^{k-1} B_{j}$ in
(\ref{e-assump cz ker low bdd}) and

\begin{equation}\label{commuassociated}
\int_{A_2^{k+1} B_{j} \setminus A_2^{k} B_{j}} |\left[b, T] f_{j}(x)\right|^{p}
\omega(x) d\mu(x) \leq \tilde{C}_{2}  \frac{  \mu( B_j)^p }{\mu(A_2^kB_j)^p } \left[\omega\left(B_{j}\right)\right]^{\kappa-1} \omega\left(A_2^{k}
B_{j}\right).
\end{equation}
\end{lem}
\begin{proof}
For each $j \in \mathbb{N},$ we define function $f_{j}$ as follows:
\[
f_{j}^{(1)}:=\chi_{B_{j, 1}}-\chi_{B_{j, 2}}:=\chi_{\left\{x \in B_{j} ; b(x)>\alpha_{B_{j}}(b)\right\}}-
\chi_{\left\{x \in B_{j} ; b(x)<\alpha_{B_{j}}(b)\right\}}, \quad f_{j}^{(2)}:=a_{j} \chi_{B_{j}}
\]
and
\[
f_{j}:=\left[\omega\left(B_{j}\right)\right]^{\frac{\kappa-1}{p}}\left(f_{j}^{(1)}-f_{j}^{(2)}\right),
\]
where $B_{j}$ is as in the assumption of Lemma \ref{lemsequence} and $a_{j} \in \mathbb{R}$ is a constant
such that

\begin{equation}\label{meanvalue}
\int_{X} f_{j}(x) d\mu(x)=0.
\end{equation}
Then, using the definition of $a_{j}$, \eqref{greater} and \eqref{lesser} we have $\left|a_{j}\right| \leq 1 / 2,
\operatorname{supp}\left(f_{j}\right) \subset B_{j}$ and, for any $x \in B_{j}$,
\begin{equation}\label{supp}
f_{j}(x)\left(b(x)-\alpha_{B_{j}}(b)\right) \geq 0.
\end{equation}
Also, since $\left|a_{j}\right| \leq 1 / 2,$ we obtain that, for any $x \in\left(B_{j, 1} \cup
B_{j, 2}\right)$,
\begin{equation}\label{sim}
\left|f_{j}(x)\right| \sim\left[\omega\left(B_{j}\right)\right]^{\frac{\kappa-1}{p}}
\end{equation}
and therefore
\begin{align*}
\left\|f_{j}\right\|_{L_{w}^{p, \kappa}\left(X\right)}
&\lesssim \sup _{B \subset X}\left\{\frac{\omega\left(B \cap B_{j}\right)}{[\omega(B)]^{\kappa}}\right\}^{\frac{1}
{p}}\left[\omega\left(B_{j}\right)\right]^{\frac{k-1}{p}} \\
& \lesssim \sup _{B \subset X}\left[\omega\left(B \cap B_{j}\right)\right]^{\frac{1-\kappa}
{p}}\left[\omega\left(B_{j}\right)\right]^{\frac{k-1}{p}} \lesssim 1 .
\end{align*}
Observe that, for any $k \in \mathbb{N},$ we get
\begin{equation}\label{inclusion}
A_2^{k-1} B_{j} \subset\left(A_2+1\right) B_{j}^{k} \subset A_2^{k+1} B_{j}
\end{equation}
hence we have
\begin{equation}\label{omegsim}
\omega\left(B_{j}^{k}\right) \sim \omega\left(A_2^{k} B_{j}\right)
\end{equation}
Observe that
\begin{equation}\label{commu}
 [b, T](f) = \left[b-\alpha_{B}(b)\right]T(f)- T([b-\alpha_{B_j}(b)]f).
\end{equation}
Using Kernel estimates, \eqref{meanvalue}, \eqref{sim} and the fact that $d\left(x, x_{j}\right) \sim d(x, \xi)$ for any $x\in B_{j}^{k}$ with integer $k \geq 2$ and $\xi \in B_{j},$ we have, for any $x \in B_{j}^{k}$,

\begin{eqnarray}
|\left[b(x)-\alpha_{B_j}(b)\right] T\left(f_{j}\right)(x)|&=&\left|b(x)-
\alpha_{B_{j}}(b)\right| \Big| \int_{B_{j}}\left[K(x, \xi)-K\left(x, x_{j}\right)\right] f_{j}(\xi) d\mu(\xi)\Big|\label{firsteq}\\
&\leq&\left|b(x)-\alpha_{B_{j}}(b)\right| \int_{B_{j}}\left|K(x, \xi)-K\left(x, x_{j}\right)\right|\left|f_{j}
(\xi)\right| d\mu(\xi) \nonumber \\
&\lesssim&\left[\omega\left(B_{j}\right)\right]^{\frac{\kappa-1}{p}}\left|b(x)-\alpha_{B_{j}}(b)\right|
\int_{B_{j}}\frac{1}{V(x,x_j)} \left(\frac{d\left(\xi, x_{j}\right)}{d\left(x, x_{j}\right)}\right)^{\sigma_0} d\mu(\xi)\nonumber \\
&\lesssim& \frac{\left[\omega\left(B_{j}\right)\right]^{\frac{\kappa-1}{p}}}{A_2^{k\sigma_0}}{\mu(B_j)\over \mu(A_2^kB_j)}\left|b(x)-
\alpha_{B_{j}}(b)\right|.\nonumber
\end{eqnarray}
As $\|b\|_{\mathrm{BMO}\left(X\right)}=1$  by John-Nirenberg
inequality(c.f.\cite{Rje}), for each $k \in \mathbb{N}$ and ball $B \subset X$,
we have
\begin{equation}\label{lowerdim}
\begin{aligned}
\int_{A_2^{k+1} B}\left|b(x)-\alpha_{B}(b)\right|^{p} d\mu(x)
&\lesssim \int_{A_2^{k+1} B}\left|b(x)-\alpha_{A_2^{k+1} B}(b)\right|^{p} d\mu(x)+\mu(A_2^{k+1}
B)\left|\alpha_{A_2^{k+1} B}(b)-\alpha_{B}(b)\right|^{p} \\
&\lesssim k^{p}\mu(A_2^{k} B),
\end{aligned}
\end{equation}
where the last inequality follows from the fact that
\[
\left|\alpha_{A_2^{k+1} B}(b)-\alpha_{B}(b)\right| \lesssim\left|\alpha_{A_2^{k+1} B}(b)-
b_{A_2^{k+1} B}\right|+\left|b_{A_2^{k+1} B}-b_{B}\right|+\left|b_{B}-\alpha_{B}(b)\right|
\lesssim k.
\]
As $\omega \in A_{p}\left(X\right),$ we observe that there exists $\epsilon \in(0, \infty)$ such that
the reverse H\"{o}lder inequality
\[
\left[\frac{1}{\mu(B)} \int_{B} \omega(x)^{1+\epsilon} d\mu( x)\right]^{\frac{1}{1+\epsilon}} \lesssim \frac{1}{\mu(B)}
\int_{B} \omega(x) d\mu(x)
\]
holds for any ball $B \subset X$. Then using the H\"{o}lder inequality, \eqref{lowerdim}, \eqref{inclusion} and \eqref{firsteq} we can obtain a positive constant $\widetilde{C}_{3}$ such that, for
any $k \in \mathbb{N}$

\begin{align}\label{commulower}
&\int_{B_{j}^{k}}\left|\left[b(x)-\alpha_{B_{j}}(b)\right] T\left(f_{j}\right)
(x)\right|^{p} \omega(x) d\mu(x)\\ \nonumber
&\lesssim \frac{\left[\omega\left(B_{j}\right)\right]^{ \kappa-1 }}{A_2^{k\sigma_0p}}{\mu(B_j)^p\over \mu(A_2^kB_j)^p} \int_{A_2^{k+1}
B_{j}}\left|b(x)-\alpha_{B_{j}}(b)\right|^{p} \omega(x) d\mu (x) \\ \nonumber
&\lesssim \frac{\left[\omega\left(B_{j}\right)\right]^{ \kappa-1 }}{A_2^{k\sigma_0p}}{\mu(B_j)^p\over \mu(A_2^kB_j)^{p-1}} \left[\frac{1}{\mu(A_2^{k+1}B_j)}\int_{A_2^{k+1}B_j}|b(x)-\alpha_{B_j}(b)|^{p(1+\epsilon)'}d\mu(x)\right]^\frac{1}{(1+\epsilon)'} \\ \nonumber
& \times \left[\frac{1}{\mu(A_2^{k+1}B_j)}\int_{A_2^{k+1}B_j}\omega(x)^{1+\epsilon}d\mu(x)\right]^\frac{1}{1+\epsilon}\\ \nonumber
&\leq \widetilde{C}_{3} \frac{k^p}{A_2^{k\sigma_0p}}{\mu(B_j)^p\over \mu(A_2^kB_j)^{p}} \left[\omega\left(B_{j}\right)\right]^{\kappa-1}
\omega\left(A_2^{k} B_{j}\right).\nonumber
\end{align}

Using Lemma \ref{lemrelcompact}, \eqref{supp}, \eqref{sim}, \eqref{Msim}, \eqref{mdel} and \eqref{lowerbound} for any $x \in B_{j}^{k}$, we get that

\[
\begin{aligned}
\left|T\left(\left[b-\alpha_{B_{j}}(b)\right] f_{j}\right)(x)\right| &=\left|\int_{B_{j, 1} \cup
B_{j, 2}} K(x, \xi)\left[b(\xi)-\alpha_{B_{j}}(b)\right] f_{j}(\xi) d \xi\right| \\
& \gtrsim \int_{B_{j, 1} \cup B_{j, 2}} \frac{\left|\left[b(\xi)-\alpha_{B_{j}}(b)\right] f_{j}(\xi)\right|}
{\mu(B(x,d(x,\xi)))} d\mu(\xi) \\
& \gtrsim \frac{1}{\mu(A_2^kB_j)}\left[\omega\left(B_{j}\right)\right]^{\frac{\kappa-1}{p}}
\int_{B_{j}}\left|b(\xi)-\alpha_{B_{j}}(b)\right| d\mu(\xi) \\
& \gtrsim \frac{\gamma \mu( B_j)}{\mu(A_2^kB_j)}\left[\omega\left(B_{j}\right)\right]^{\frac{\kappa-1}{p}}.
\end{aligned}
\]
Along with \eqref{omegsim} we deduce that there exists a positive constant $\widetilde{C}_{4}$ such that

\begin{align}\label{ctil}
\int_{B_j^{k}}\left|T\left(\left[b-\alpha_{B_{j}}(b)\right] f_{j}\right)(x)\right|^{p} \omega(x) d\mu(x) &
\geq \frac{\gamma^p \mu( B_j)^p}{\mu(A_2^kB_j)^p}\left[\omega\left(B_{j}\right)\right]^{\kappa-1} \omega\left(B_{j}^{k}\right) \\ \nonumber
& \geq \widetilde{C}_{4} \frac{\gamma^p \mu( B_j)^p}{\mu(A_2^kB_j)^p}\left[\omega\left(B_{j}\right)\right]^{\kappa-1}
\omega\left(A_2^{k} B_{j}\right). \nonumber
\end{align}
Now let us take $K_{0} \in(0, \infty)$ large enough such that, for any integer $k \geq K_{0}$
\[
\tilde{C}_{4} \frac{\gamma^{p}}{2^{p-1}}-\widetilde{C}_{3} \frac{k^{p}}{A_2^{k\sigma_0 p}} \geq
\widetilde{C}_{4} \frac{\gamma^{p}}{2^{p}}.
\]
Using this and \eqref{commu}, \eqref{commulower} and \eqref{ctil}, we get
\begin{align*}
&\int_{B_{j}^{k}} \left|[ b, T]f_{j}(x)\right|^{p} \omega(x) d\mu(x) \\
&\quad \geq \frac{1}{2^{p-1}} \int_{B_{j}^{k}}\left|T\left(\left[b-\alpha_{B_{j}}(b)\right]f_{j}\right)(x)\right|^{p} \omega(x) d\mu(x)-\int_{B_{j}^{k}}\left|\left[b(x)-\alpha_{B_{j}}(b)\right]T\left(f_{j}\right)(x)\right|^{p} \omega(x) d\mu(x) \\
&\quad \geq\left(\tilde{C}_{4} \frac{\gamma^{p}}{2^{p-1}}-\widetilde{C}_{3} \frac{k^{p}}{A_2^{k\sigma_0 p}}\right)
\frac{  \mu( B_j)^p}{\mu(A_2^kB_j)^p}\left[\omega\left(B_{j}\right)\right]^{\kappa-1}
\omega\left(A_2^{k} B_{j}\right) \\
&\quad \geq \widetilde{C}_{4} \frac{\gamma^{p}}{2^{p}} \frac{  \mu( B_j)^p}{\mu(A_2^kB_j)^p}\left[\omega\left(B_{j}\right)\right]^{\kappa-1}
\omega\left(A_2^{k} B_{j}\right).
\end{align*}
This implies \eqref{communorm}.

Also, since $\operatorname{supp}\left(f_{j}\right) \subset B_{j},$ by \eqref{sim} and \eqref{Msim} and
$\|b\|_{\mathrm{BMO}\left(X\right)}= 1$, we deduve that, for any $x \in A_2^{k+1} B_{j} \setminus A_2^{k} B_{j}$
\[
\left|T\left(\left[b-\alpha_{B_{j}}(b)\right] f_{j}\right)(x)\right|
\lesssim\left[\omega\left(B_{j}\right)\right]^{\frac{\kappa-1}{p}} \int_{B_{j}} \frac{\left|b(\xi)-\alpha_{B_{j}}
(b)\right|}{V(x, \xi) } d\mu(\xi) \lesssim\left[\omega\left(B_{j}\right)\right]^{\frac{\kappa-1}{p}}  \frac{  \mu( B_j) }{\mu(A_2^kB_j) }.
\]
Therefore, by \eqref{commulower} with $B_{j}^{k}$ replaced by $A_2^{k+1} B_{j} \setminus A_2^{k} B_{j},$
we can deduce that, for any integer $k \geq K_{0}$
\begin{align*}
&\int_{A_2^{k+1} B_{j} \setminus A_2^{k} B_{j}}
 |\left[b, T] f_{j}(x)\right|^{p} \omega(x) d\mu (x) \\
& \lesssim  \int_{A_2^{k+1} B_{j} \setminus A_2^{k} B_{j}}\left|T\left(\left[b-
\alpha_{B_{j}}(b)\right] f_{j}\right)(x)\right|^{p} \omega(x) d\mu(x) \\
&\quad +\int_{A_2^{k+1} B_{j} \setminus A_2^{k} B_{j}}\left|\left[b(x)-\alpha_{B_{j}}(b)\right]
T\left(f_{j}\right)(x)\right|^{p} \omega(x) d\mu(x) \\
&\lesssim\frac{  \mu( B_j)^p }{\mu(A_2^kB_j)^p }\left[\omega\left(B_{j}\right)\right]^{\kappa-1} \omega\left(A_2^{k}
B_{j}\right)+  \frac{k^p}{A_2^{k\sigma_0p}}{\mu(B_j)^p\over \mu(A_2^kB_j)^{p}} \left[\omega\left(B_{j}\right)\right]^{\kappa-1}
\omega\left(A_2^{k} B_{j}\right) \\
&\lesssim \frac{  \mu( B_j)^p }{\mu(A_2^kB_j)^p }\left[\omega\left(B_{j}\right)\right]^{\kappa-1} \omega\left(A_2^{k}
B_{j}\right).
\end{align*}
This completes the proof of Lemma \ref{lemsequence}.
\end{proof}

The following technical result are needed to handle the weighted estimate for  showing the necessity of the compactness of the commutators.

\begin{lem}\label{weightedestimate}
Let $ 1<p<\infty, 0<\kappa<1, \omega \in A_{p}\left(X\right), b \in\mathrm{BMO}\left(X\right), \gamma, K_{0}>0,$
$\left\{f_{j}\right\}_{j \in \mathbb{N}}$ and $\left\{B_{j}\right\}_{j \in \mathbb{N}}$ be as given in Lemma \ref{lemsequence}.
Now assume that $\left\{B_{j}\right\}_{j \in \mathbb{N}}:=\left\{B\left(x_{j}, r_{j}\right)\right\}_{j \in\mathbb{N}}$ also satisfies the following two conditions:
 \begin{enumerate}
\item[(i)] $\forall \ell, m \in \mathbb{N}$ and $\ell \neq m$

\begin{equation}\label{intersection}
A_2 C_{1} B_{\ell} \bigcap A_2 C_{1} B_{m}=\emptyset,
\end{equation}

where $C_{1}:=A_2^{K_{1}}>C_{2}:=A_2^{K_{0}}$ for some $K_{1} \in \mathbb{N}$ large enough.
\item [(ii)] $\left\{r_{j}\right\}_{j \in \mathbb{N}}$ is either non-increasing or non-decreasing in $j$, or there
exist positive constants $C_{\min }$ and $C_{\max }$ such that, for any $j \in \mathbb{N}$
\[
C_{\min } \leq r_{j} \leq C_{\max }.
\]
Then there exists a positive constant $C$ such that, for any $j, m \in \mathbb{N}$
\[
\left\|[b, T] f_{j}-[b, T] f_{j+m}\right\|_{L_{\omega}^{p, \kappa}\left(X\right)} \geq C.
\]
\end{enumerate}
\end{lem}

\begin{proof}
Without loss of generality, we assume that $\|b\|_{\mathrm{BMO}\left(X\right)}=1$
and $\left\{r_{j}\right\}_{j \in \mathrm{N}}$ is non-increasing. Let $\left\{f_{j}\right\}_{j \in \mathrm{N}}, \widetilde{C}_{1}, \widetilde{C}_{2}$ be as in Lemma \ref{lemsequence} associated with $\left\{B_{j}\right\}_{j \in
\mathbb{N}}$.

By \eqref{communorm}, \eqref{omegsim},
Lemma \ref{lemcomparison} with $\omega \in A_{p }\left(X\right)$, we observe that, for any $j \in \mathbb{N}$,

\begin{align}\label{uppebound}
&\left[\int_{A_{2}^{K_0} B_{j}} |[b, T]f_{j}(x)|^{p} \omega(x) d\mu (x)\right]^{1 /p}\left[\omega\left(A_{2}^{K_0} B_{j}\right)\right]^{-\kappa / p}\\ \nonumber
&\geq\left[\omega\left(A_{2}^{K_0} B_{j}\right)\right]^{-\kappa / p}\left\{\int_{B_{j}^{K_{0}-1}} |\left[b,T] f_{j}(x)\right|^{p} \omega(x) d\mu(x)\right\}^{1 / p}\\ \nonumber
&\geq\left[\omega\left(A_{2}^{K_0} B_{j}\right)\right]^{-\kappa / p}\left\{\widetilde{C}_{1} \gamma^{p}{ \frac{  \mu( B_j )^p}{\mu(A_2^{K_{0}-1}B_j )^p}}\left[\omega\left(B_{j}\right)\right]^{\kappa-1} \omega\left(A_2^{K_{0}-1}B_{j}\right)\right\}^{1 / p}\\ \nonumber
&\gtrsim\left[\omega\left(A_{2}^{K_0} B_{j}\right)\right]^{-\kappa / p}\left\{\gamma^{p}
\frac{\left[\omega\left(B_{j}\right)\right]^{\kappa}}{A_{2}^{np\left(K_{0}-1\right)}}\right\}^{1 / p}\\ \nonumber
&\geq C_{3}\gamma A_{2}^{- {n (\kappa K_0+K_0-1)} }\left[\omega\left(B_{j}\right)\right]^{-\kappa / p}
 \left[\omega\left(B_{j}\right)\right]^{\kappa / p}\\ \nonumber
&=C_{3}\gamma A_{2}^{- {n (\kappa K_0+K_0-1)} } \nonumber
\end{align}
holds for a positive constant $C_{3}$ independent of $\gamma$ and $A_{2}$. We also show that, for any $j, m \in \mathbb{N}$,

\begin{equation}\label{commubound}
\left[\int_{A_{2}^{K_0} B_{j}}\left|[b,T] f_{j+m}(x)\right|^{p} \omega(x) d\mu(x)\right]^{1 / p}\left[\omega\left(A_{2}^{K_0}
B_{j}\right)\right]^{-\kappa / p} \leq \frac{1}{2} C_{3} \gamma A_{2}^{- {n (\kappa K_0+K_0-1)} }.
\end{equation}
As $\operatorname{supp}\left(f_{j+m}\right) \subset B_{j+m}$, from \eqref{Msim}, \eqref{sim}, \eqref{intersection} and
$\|b\|_{\mathrm{BMO}\left(X\right)}=1,$ it
follows that, for any $x \in A_{2}^{K_0} B_{j}  $

\begin{eqnarray*}
&\left|T\left(\left[b-\alpha_{B_{j+m}}(b)\right] f_{j+m}\right)(x)\right|
\lesssim\left[\omega\left(B_{j+m}\right)\right]^{\frac{\kappa-1}{p}} \int_{B_{j+m}}|K(x, \xi)|\left|b(x)-
\alpha_{B_{j+m}}(b)\right| d\mu(\xi)\\
&\lesssim \left[\omega\left(B_{j+m}\right)\right]^{\frac{\kappa-1}{p}} \frac{\mu(B_{j+m})}{V\left(x_j,x_{j+m}\right)}.
\end{eqnarray*}
So we have

\begin{equation}\label{tsim}
\left\{\int_{A_{2}^{K_0} B_{j}}\left|T\left(\left[b-\alpha_{B_{j+m}}(b)\right] f_{j+m}\right)(x)\right|^{p}
\omega(x) d\mu(x)\right\}^\frac{1}{p}\left[\omega\left(A_{2}^{K_0} B_{j}\right)\right]^{-\frac{\kappa}{p}}
\end{equation}

\begin{align*}
&\lesssim\left[\omega\left(B_{j+m}\right)\right]^{\frac{\kappa-1}{p}} \frac{\mu(B_{j+m})}{V\left(x_{j},
x_{j+m}\right)}\left[\omega\left(A_{2}^{K_0} B_{j}\right)\right]^{\frac{1-\kappa}{p}}.
\end{align*}
Also, using \eqref{sim} we deduce that, for any $x \in A_{2}^{K_0} B_{j}$

\begin{equation}\label{ratio}
\begin{aligned}
\left|T\left(f_{j+m}\right)(x)\right| & \leq \int_{B_{j+m}}\left|K(x, \xi)-K\left(x,
x_{j+m}\right)\right|\left|f_{j+m}(\xi)\right| d\mu(\xi) \\
& \lesssim\left[\omega\left(B_{j+m}\right)\right]^{\frac{\kappa-1}{p}}{\mu(B_{j+m})\over V(x_j,x_{j+m})} \frac{r_{j+m}^{\sigma_0}}{d\left(x_{j},
x_{j+m}\right)^{\sigma_0}}.
\end{aligned}
\end{equation}
Hence, using \eqref{ratio} and the fact $\{r_j\}_{j\in \mathbb{N}}$ is non-increasing in $j$ and from H\"{o}lders and reverse H\"{o}lder inequalities we get that
\begin{align*}
&\left\{\int_{A_{2}^{K_0} B_{j}}\left|\left[b(x)-\alpha_{B_{j+m}}(b)\right] T\left(f_{j+m}\right)
(x)\right|^{p} \omega(x) d\mu(x)\right\}^{1 / p}\left[\omega\left(A_{2}^{K_0} B_{j}\right)\right]^{-\kappa / p}\\
&\lesssim\left[\omega\left(B_{j+m}\right)\right]^{\frac{\kappa-1}{p}} {\mu(B_{j+m})\over V(x_j,x_{j+m})} \frac{r_{j+m}^{\sigma_0}}{d\left(x_{j},
x_{j+m}\right)^{\sigma_0}}\left[\omega\left(A_{2}^{K_0} B_{j}\right)\right]^{-\kappa / p} \\
&\qquad\times \left[\int_{A_{2}^{K_0} B_{j}} \left|b(x)-
\alpha_{B_{j+m}}(b)\right|^{p} \omega(x) d\mu(x)\right]^{1 / p}\\
&\lesssim\left[\omega\left(B_{j+m}\right)\right]^{\frac{\kappa-1}{p}} {\mu(B_{j+m})\over V(x_j,x_{j+m})} \frac{r_{j+m}^{\sigma_0}}{d\left(x_{j},
x_{j+m}\right)^{\sigma_0}}\left[\omega\left(A_{2}^{K_0} B_{j}\right)\right]^{\frac{1-\kappa}{p}} \\
&\qquad\times\left(\log \frac{d\left(x_{j},
x_{j+m}\right)}{r_{j+m}}+\log \frac{d\left(x_{j}, x_{j+m}\right)}{r_{j}}\right)\\
&\lesssim\left[\omega\left(B_{j+m}\right)\right]^{\frac{\kappa-1}{p}}{\mu(B_{j+m})\over V(x_j,x_{j+m})} \left[\omega\left(A_{2}^{K_0} B_{j}\right)\right]^{\frac{1-\kappa}{p}} \frac{r_{j+m}^{\sigma_0}}{d\left(x_{j},
x_{j+m}\right)^{\sigma_0}}\log \frac{d\left(x_{j}, x_{j+m}\right)}{r_{j+m}}.
\end{align*}
Observe that, for $C_{1}$ large enough, using \eqref{intersection} we know that $d\left(x_{j}, x_{j+m}\right)$ is also large enough and so we have
\begin{equation}\label{lessratio}
 \quad\left(\frac{d\left(x_{j},x_{j+m}\right)}{r_{j+m}}\right)^{-\sigma_0} \log \frac{d\left(x_{j},x_{j+m}\right)}{r_{j+m}} \lesssim 1.
\end{equation}
Using \eqref{tsim}, \eqref{ratio} and \eqref{lessratio}, we obtain that

\begin{align*}
&\left\{\int_{A_{2}^{K_0} B_{j}}\left|[b,T]\left(f_{j+m}\right)(x)\right|^{p} \omega(x) d\mu(x)\right\}^{1 /p}\left[\omega\left(A_{2}^{K_0} B_{j}\right)\right]^{-\kappa / p}\\
&\leq\left\{\int_{A_{2}^{K_0} B_{j}}\left|T\left(\left[b-\alpha_{B_{j+m}}(b)\right] f_{j+m}\right)
(x)\right|^{p} \omega(x) d\mu(x)\right\}^{1 / p}\left[\omega\left(A_{2}^{K_0} B_{j}\right)\right]^{-\kappa / p}\\
&+\left\{\int_{A_{2}^{K_0} B_{j}}\left|\left[b(x)-\alpha_{B_{j+m}}(b)\right] T\left(f_{j+m}\right)
(x)\right|^{p} \omega(x) d\mu(x)\right\}^{1 / p}\left[\omega\left(A_{2}^{K_0} B_{j}\right)\right]^{-\kappa / p}\\
&\lesssim \left[\omega\left(B_{j+m}\right)\right]^{\frac{\kappa-1}{p}}{\mu(B_{j+m})\over V(x_j,x_{j+m})} \left[\omega\left(A_{2}^{K_0} B_{j}\right)\right]^{\frac{1-\kappa}{p}} \\
&\lesssim  {\mu(B_{j+m})\over V(x_j,x_{j+m})} \left[{\omega\left(B(x_j,d(x_j,x_{j+m}))\right)\over \omega\left(B_{j+m}\right)}\right]^{\frac{1-\kappa}{p}} \\
&\leq   C'\left[{\mu(B_{j+m})\over V(x_j,x_{j+m})}\right]^{\kappa} .
\end{align*}
Note that $\lim\limits_{k\to \infty}\mu(A_2^kB_{j+m})=\infty.$ Then for  $C_{1}$ large enough, we have $$\mu(C_1B_{j+m})\geq \Big({2C'\over C_{3}\gamma A_{2}^{- {n (\kappa K_0+K_0-1)} }}\Big)^{1\over \kappa}\mu( B_{j+m}).$$
This implies that $C'\left[{\mu(B_{j+m})\over V(x_j,x_{j+m})}\right]^{\kappa}\leq C'\left[{\mu(B_{j+m})\over \mu(C_1B_{j+m})}\right]^{\kappa}\leq {1\over2}C_{3}\gamma A_{2}^{- {n (\kappa K_0+K_0-1)} }.$
This gives the proof of \eqref{commubound}. Using \eqref{lowerbound} and \eqref{commubound} we know that, for any $j, m \in \mathbb{N}$ and $C_{1}$ large enough

\begin{align*}
&\left\{\int_{A_{2}^{K_0} B_{j}}\left|[b,T]\left(f_{j}\right)(x)-[b,T]\left(f_{j+m}\right)
(x)\right|^{p} \omega(x) d\mu(x)\right\}^{1 / p}\left[\omega\left(A_{2}^{K_0} B_{j}\right)\right]^{-\kappa / p} \\
&\geq\left\{\int_{A_{2}^{K_0} B_{j}}\left|[b,T]\left(f_{j}\right)(x)\right|^{p} \omega(x) d\mu(x)\right\}^{1 /
p}\left[\omega\left(A_{2}^{K_0} B_{j}\right)\right]^{-\kappa / p} \\
&\quad-\left\{\int_{A_{2}^{K_0} B_{j}}\left|[b,T]\left(f_{j+m}\right)(x)\right|^{p} \omega(x) d\mu(g)\right\}^{1 /
p}\left[\omega\left(A_{2}^{K_0} B_{j}\right)\right]^{-\kappa / p} \geq {1\over2}C_{3}\gamma A_{2}^{- {n (\kappa K_0+K_0-1)} }.
\end{align*}
This completes the proof of Lemma 4.7.
\end{proof}

\begin{proof}[Proof of Theorem \ref{thm main2}(ii)]
 Without loss of generality, we assume that
$\|b\|_{\mathrm{BMO}\left(X\right)}= 1$. To prove $b \in \mathrm{VMO}\left(X\right)$, observe that $b \in
\mathrm{BMO}\left(X\right)$ is a real-valued function, we  will use a contradiction argument via Lemmas \ref{lemvmo}, \ref{lemsequence} and \ref{weightedestimate}.  Now note that, if $b \notin\mathrm{VMO}\left(X\right),$ then $b$ does not satisfy at least one of (i) through (iii)
of Lemma \ref{lemvmo}. We show that $[b, T]$ is not compact on $L_{\omega}^{p, \kappa}\left(X\right)$ in any of the following three cases.

\textbf{Case (i)} $b$ does not satisfy condition (i)  Lemma \ref{lemvmo}. Hence there exist $\gamma \in(0, \infty)$ and a sequence
\[
\left\{B_{j}^{(1)}\right\}_{j \in \mathbb{N}}:=\left\{B(x_{j}^{(1)}, r_{j}^{(1)})\right\}_{j \in\mathbb{N}}
\]
of balls in $X$ satisfying \eqref{mdel} and that $r_{j}^{(1)} \rightarrow 0$ as $j \rightarrow
\infty$. Let $x_0$ be a fixed point in $X$. We now consider the following two subcases.

\textbf{Subcase (i)} There exists a positive constant $M$ such that $0 \leq d(x_0,x_{j}^{(1)})<M$ for all $x_{j}^{(1)}, j \in \mathbb{N}$. That is, $x_{j}^{(1)} \in B_{0}:=B(x_0, M), \forall j \in \mathbb{N}$. Let $\left\{f_{j}\right\}_{j \in \mathbb{N}}$ be associated with the sequence $\left\{B_{j}\right\}_{j \in \mathrm{N}}, \tilde{C}_{1}$
$\tilde{C}_{2}, K_{0}$ and $C_{2}$ be as in Lemmas \ref{lemsequence} and \ref{weightedestimate}. Let $p_{0}\in(1, p)$ be such that
$\omega \in A_{p_{0}}\left(X\right)$ and $C_{4}:=A_2^{K_{2}}>C_{2}=A_2^{K_{0}}$ for
$K_{2} \in \mathbb{N}$ large enough so that

\begin{equation}\label{c5}
C_{5}:={\widetilde{C}_{1} \hat{C}_{2} \gamma^{p}\over C_{\mu}}  {A_2}^{n K_{0}(\sigma-p) }>2 {\widetilde{C}_{2}\over \hat{C}_1}
  {A_2^{K_2( p_{0}-p) }\over 1-A_2^{K_2( p_{0}-p)}} ,
\end{equation}
where $\hat{C}_{1}$ and $\hat{C}_{2}$ are as in Lemma \ref{lemcomparison}. As we know $\left|r_{j}^{(1)}\right| \rightarrow 0$ as $j \rightarrow \infty$ and $\left\{x_{j}^{(1)}\right\}_{j \in \mathbb{N}} \subset B_{0}$, we choose a subsequence $\left\{B_{j_{\ell}}^{(1)}\right\}_{\ell \in \mathbb{N}}$ of $\left\{B_{j}^{(1)}\right\}_{j \in\mathbb{N}}$ so that, for any $j \in \mathbb{N}$,

\begin{equation}\label{sizeratios}
\frac{\mu\left(B_{j_{\ell+1}}^{(1)}\right)}{\mu\left(B_{j_{\ell}}^{(1)}\right)}<\frac{1}{C_{4}^{n}} \text { and }
\omega\left(B_{j_{\ell+1}}^{(1)}\right) \leq \omega\left(B_{j_\ell}^{(1)}\right).
\end{equation}
Define for any fixed $\ell, m \in \mathbb{N}$
\[
\mathcal{J}:=C_{4} B_{j_{\ell}}^{(1)} \backslash C_{2} B_{j_{\ell}}^{(1)}, \quad
\mathcal{J}_{1}:=\mathcal{J} \backslash C_{4} B_{j_{\ell+m}}^{(1)} \quad \text { and } \quad
\mathcal{J}_{2}:=X \backslash C_{4} B_{j_{\ell+m}}^{(1)}.
\]
Observe that
\[
\mathcal{J}_{1} \subset\left[\left(C_{4} B_{j\ell}^{(1)}\right) \cap \mathcal{J}_{2}\right] \quad \text { and }
\quad \mathcal{J}_{1}=\mathcal{J} \cap \mathcal{J}_{2}.
\]
Hence we have

\begin{align}\label{setdifference}
&\left\{\int_{C_{4} B_{j_{\ell}}^{(1)}}\left|[b,T]\left(f_{j_{\ell}}\right)(x)-[b,T]\left(f_{j_{i+m}}\right)(x)\right|^{p} \omega(x) d\mu (x)\right\}^{1 /p}\\ \nonumber
&\geq\left\{\int_{\mathcal{J}_{1}}\left|[b,T]\left(f_{j_{\ell}}\right)(x)-[b,T]\left(f_{j_{\ell+m}}\right)(x)\right|^{p} \omega(x) d\mu(x)\right\}^{1 / p}\\ \nonumber
&\geq\left\{\int_{\mathcal{J}_{1}}\left|[b,T]\left(f_{j_{\ell}}\right)(x)\right|^{p} \omega(x) d\mu(x)\right\}^{1 /
p}-\left\{\int_{\mathcal{J}_{2}}\left|[b,T]\left(f_{j_{\ell+m}}\right)(x)\right|^{p} \omega(x) d\mu(x)\right\}^{1
/ p} \\ \nonumber
&=\left\{\int_{\mathcal{J} \cap \mathcal{J}_{2}}\left|[b, T]\left(f_{j_{\ell}}\right)(x)\right|^{p} \omega(x)
d\mu(x)\right\}^{1 / p}-\left\{\int_{\mathcal{J}_{2}}\left|[b,T]\left(f_{j_{\ell+m}}\right)(x)\right|^{p}
\omega(x) d\mu(x)\right\}^{1 / p} \\ \nonumber
&=: \mathrm{F}_{1}-\mathrm{F}_{2}. \nonumber
\end{align}
First we consider the term $\mathrm{F}_{1}$ and assume that $E_{j_{\ell}}:=\mathcal{J} \backslash
\mathcal{J}_{2} \neq \emptyset$. Then $E_{j_{\ell}} \subset C_{4} B_{j_{\ell+m}}^{(1)}$
by \eqref{sizeratios}  we have
\begin{equation}\label{inclusionin}
\mu\left(E_{j_{\ell}}\right) \leq C_{4}^{n}\mu\left(B_{j_{\ell+m}}^{(1)}\right)<\mu\left(B_{j_{\ell}}^{(1)}\right).
\end{equation}
Now take
\[
B_{j_{\ell,k}}^{(1)}:= \widetilde{A_2^{k-1} B_{j_{\ell}}^{(1)}},
\]
to be the ball associates with $A_2^{k-1} B_{j_{\ell}}^{(1)}$ in (\ref{e-assump cz ker low bdd}).
Now using \eqref{inclusionin}, we get
\[
\mu\left(B_{j_{\ell,k}}^{(1)}\right)= \mu\left(A_2^{k-1}B_{j_{\ell}}^{(1)}\right)> \mu(E_{j_{\ell}}).
\]
Using this, we further know that there exist finite $\left\{B_{j_{\ell,k}}^{(1)}\right\}_{k=K_{0}}^{K_{2}-2}$
intersecting $E_{j_{\ell}}$. Then, from \eqref{communorm} and Lemma \ref{lemcomparison}, we deduce that

\begin{align}\label{f1lowerbound}
\mathrm{F}_{1}^{p} &\geq \sum_{k=K_{0}, B_{j_{\ell}, k}^{(1)} \cap E_{j_{\ell}}=\emptyset}^{K_{2}-2} \int_{B_{j_{\ell,k}}^{(1)}}\left|[b,T]\left(f_{j_{\ell}}\right)(x)\right|^{p}\omega(x) d\mu(x) \\ \nonumber
& \geq \widetilde{C}_{1} \gamma^{p} \sum_{k=K_{0}, B_{j_{\ell,k}}^{(1)} \cap E_{j_{\ell}}=\emptyset}^{K_{2}-2}  {\mu\Big(B_{j_{\ell}}^{(1)}\Big)^p\over \mu\Big(A_2^{k} B_{j_{\ell}}^{(1)}\Big)^p}{ \omega\left(B_{j_{\ell}}^{(1)}\right)^{\kappa-
1} \omega\left(A_2^{k} B_{j_{\ell}}^{(1)}\right) } \\ \nonumber
&\geq\sum_{k=K_{0}, B_{j_{\ell,k}}^{(1)} \cap E_{j_{\ell}}=\emptyset}^{K_{2}-2} \widetilde{C}_{1}\hat{C}_{2} \gamma^{p}{\mu\Big(B_{j_{\ell}}^{(1)}\Big)^p\over \mu\Big(A_2^{k} B_{j_{\ell}}^{(1)}\Big)^p}{A_2^{n k\sigma}}\omega\left(B_{j_{\ell}}^{(1)}\right)^{\kappa} \\ \nonumber
& \geq{\widetilde{C}_{1} \hat{C}_{2} \gamma^{p}\over C_{\mu}}  {A_2}^{n K_{0}(\sigma-p) } \omega\left(B_{j_{\ell}}^{(1)}\right) ^{\kappa}=C_{5}\omega\left(B_{j_{\ell}}^{(1)}\right) ^{\kappa}. \nonumber
\end{align}

If $E_{j_l} := \mathcal{J}\setminus\mathcal{J}_2 = \emptyset$, the inequality is still holds true.

Observe that  $\lim\limits_{k\to \infty}\mu(A_2^kB_{j_{l+m}}^{(1)})=\infty$. Then  there exist $j_k\in \mathbb{N}$ such that
$$\mu(A_2^{j_1}B_{j_{l+m}}^{(1)}) \geq A_2^{K_2}\mu(A_2^{K_2}B_{j_{l+m}}^{(1)})\  {\rm and }\ \mu(A_2^{j_{k+1}}B_{j_{l+m}}^{(1)})\geq A_2^{K_2}\mu(A_2^{j_k}B_{j_{l+m}}^{(1)}).$$
Also, using the proof of  \eqref{commuassociated}, Lemma \ref{lemsequence}, \eqref{c5} and \eqref{sizeratios}, we obtain that

\begin{align}\label{f2lower}
\mathrm{F}_{2}^{p} &\leq \sum_{k=0}^{\infty} \int_{A_2^{j_{k+1}}B_{j_{l+m}}^{(1)}\setminus A_2^{j_{k}}B_{j_{l+m}}^{(1)}}|[b,T](f_{j_{l+m}})(x)|^p\omega(x)d\mu(x)\\ \nonumber
&\leq \widetilde{C}_{2}\sum_{k=0}^{\infty}  {\mu\Big(B_{j_{\ell+m}}^{(1)}\Big)^p\over \mu\Big(A_2^{j_{k}} B_{j_{\ell+m}}^{(1)}\Big)^p}{\left[\omega\left(B_{j_{\ell+m}}^{(1)}\right)\right]^{\kappa-1} \omega\left(A_2^{j_{k}}
B_{j_{\ell+m}}^{(1)}\right)}  \\ \nonumber
&\leq \widetilde{C}_{2}\sum_{k=0}^{\infty}  {\mu\Big(B_{j_{\ell+m}}^{(1)}\Big)^p\over A_2^{(k+1)K_2p}\mu\Big(A_2^{K_2} B_{j_{\ell+m}}^{(1)}\Big)^p}{\left[\omega\left(B_{j_{\ell+m}}^{(1)}\right)\right]^{\kappa-1} {1\over \hat{C}_1}{A_2^{(k+1)K_2p_0}\mu\Big(A_2^{K_2} B_{j_{\ell+m}}^{(1)}\Big)^{p_0}\over \mu\Big(B_{j_{\ell+m}}^{(1)}\Big)^{p_0}}\omega\left(
B_{j_{\ell+m}}^{(1)}\right)}  \\ \nonumber
& \leq {\widetilde{C}_{2}\over \hat{C}_1}  \sum_{k=0}^{\infty}
  {A_2^{(k+1)K_2( p_{0}-p) }} {\mu\Big(B_{j_{\ell+m}}^{(1)}\Big)^{p-p_0}\over \mu\Big(A_2^{K_2} B_{j_{\ell+m}}^{(1)}\Big)^{p-p_0}}\left[\omega\left(B_{j_{\ell+m}}^{(1)}\right)\right]^{\kappa }  \\ \nonumber
  & \leq {\widetilde{C}_{2}\over \hat{C}_1}
  {A_2^{K_2( p_{0}-p) }\over 1-A_2^{K_2( p_{0}-p)}}  \left[\omega\left(B_{j_{\ell+m}}^{(1)}\right)\right]^{\kappa }.  \\ \nonumber
\end{align}
By \eqref{sizeratios}, \eqref{setdifference},\eqref{f1lowerbound} and \eqref{f2lower} we deduce
\[
\begin{array}{l}
\left\{\int_{C_{4} B_{j_{\ell}}^{(1)}}\left|[b,T]\left(f_{j_{\ell}}\right)(x)-[b,T]\left(f_{j_{\ell+m}}\right)(x)\right|^{p} \omega(x) d\mu (x)\right\}^{1 / p} \\
\geq C_{5}^{1 / p}\left[\omega\left(B_{j_{\ell}}^{(1)}\right)\right]^{\kappa / p}-\left(\frac{C_{5}}{2}\right)^{1 /
p}\left[\omega\left(B_{j_{\ell}}^{(1)}\right)\right]^{\kappa / p}\gtrsim\left[\omega\left(B_{j_{\ell}}^{(1)}\right)\right]^{\kappa / p}.
\end{array}
\]
Hence we get, $\left\{[b,T] f_{j}\right\}_{j \in \mathbb{N}}$ is not relatively compact in $L_{\omega}^{p,\kappa}\left(X\right),$ which implies that $[b,T]$ is not compact on
$L_{\omega}^{p, \kappa}\left(X\right)$. So, $b$ satisfies condition (i) of Lemma \ref{lemvmo}.

\textbf{Subcase (ii)} There exists a subsequence $\left\{B_{j_{e}}^{(1)}\right\}_{\ell \in \mathbb{N}}:=\left\{B\left(x_{j_{\ell}}^{(1)}, r_{j_{\ell}}^{(1)}\right)\right\}_{\ell \in \mathbb{N}}$ of $\left\{B_{j}^{(1)}\right\}_{j\in \mathbb{N}}$
such that $d(x_0,x_{j_{\ell}}^{(1)})  \rightarrow \infty$ as $\ell \rightarrow \infty$. In this subcase, by
$\mu\left(B_{j \ell}^{(1)}\right) \rightarrow 0$ as $\ell \rightarrow \infty,$ we take a mutually disjoint
subsequence of $\left\{B_{j_{\ell}}^{(1)}\right\}_{\ell \in \mathrm{N}}$, and denote by $\left\{B_{j_{i}}^{(1)}\right\}_{\ell \in \mathrm{N}}$, satisfying \eqref{intersection} as well. This, via Lemma \ref{weightedestimate} implies
that $[b, T]$ is not compact on $L_{\omega}^{p, \kappa}\left(X\right),$ which is a
contradiction to our assumption. Hence, $b$ satisfies condition (i) of Lemma \ref{lemvmo}.

\textbf{Case (ii)} If $b$ does not satisfy condition (ii) of Lemma \ref{lemvmo}. In this case, there exist $\gamma \in$ $(0,\infty)$ and a sequence $\left\{B_{j}^{(2)}\right\}_{j \in \mathbb{N}}$ of balls in $X$ satisfying \eqref{mdel} and that $ |r_{B_{j}^{(2)}} | \rightarrow \infty$ as $j \rightarrow \infty$. We also consider the following two subcases as well.

\textbf{Subcase (i)} There exists an infinite subsequence $\left\{B_{j \ell}^{(2)}\right\}_{\ell \in \mathbb{N}}$ of $\left\{B_{j}^{(2)}\right\}_{j \in \mathbb{N}}$ and a point $x_{0} \in X$ such that, for any $\ell
\in \mathbb{N}, x_{0} \in  A_2 C_{1} B_{j_{\ell}}^{(2)}$. As $ |r_{B_{j_{\ell}}^{(2)}} |\rightarrow \infty$ as $\ell \rightarrow \infty$, it follows
that there exists a subsequence, denoted as earlier by $\left\{B_{j \ell}^{(2)}\right\}_{\ell \in \mathrm{N}},$
such that, for any $\ell \in \mathbb{N}$

\begin{equation}\label{measureratio}
\frac{\mu\left(B_{j_{e}}^{(2)}\right)}{\mu\left(B_{j_{\ell}+1}^{(2)}\right)}<\frac{1}{C_{4}^{n}}.
\end{equation}
Note that $2 A_{2} C_{1} B_{j_{\ell}}^{(2)} \subset 2 A_{2} C_{1} B_{j_{\ell+1}}^{(2)}$ for any
$j_{\ell} \in \mathbb{N}$ and hence

\begin{equation}\label{weightinequality}
\omega\left(2  A_2 C_{1} B_{j_{\ell+1}}^{(2)}\right) \geq \omega\left(2  A_2 C_{1} B_{j_{\ell}}^{(2)}\right), \quad
M\left(b, 2  A_2 C_{1} B_{j_{\ell}}\right)>\frac{\gamma}{8 A_{2}^{2} C_{1}^{2}}.
\end{equation}
Using similar method as that used in Subcase (i) of Case (i) and we redefine our sets in a reversed order. That is, for any fixed $\ell, k \in \mathbb{N},$ we let

\[
\begin{aligned}
\widetilde{\mathcal{J}} &:=2  A_2 C_{4} C_{1} B_{\ell+k}^{(2)} \backslash 2  A_2 C_{2} C_{1}
B_{\ell+k}^{(2)}, \\
\widetilde{\mathcal{J}}_{1} &:=\widetilde{\mathcal{J}} \backslash 2  A_2 C_{4} C_{1} B_{j_l}^{(2)}, \\
\widetilde{\mathcal{J}}_{2} &:= X \backslash 2  A_2 C_{4} C_{1} B_{j_{\ell}}^{(2)}.
\end{aligned}
\]
As in Case (i), by Lemma \ref{lemsequence}, \eqref{measureratio} and \eqref{weightinequality}, we deduce that the commutator $[b, T]$ is not compact on $L_{\omega}^{p, \kappa}\left(X\right) .$ This contradiction gives that $b$
satisfies condition (ii) of Lemma \ref{lemsequence}.

\textbf{Subcase (ii)} For any $z \in X$ the number of $\left\{ A_2 C_{1} B_{j}^{(2)}\right\}_{j \in
\mathrm{N}}$ containing $z$ is finite. In this subcase, for each square $B_{j_{0}}^{(2)} \in\left\{B_{j}^{(2)}\right\}_{j \in \mathrm{N}},$ the number of $\left\{ A_2 C_{1} B_{j}^{(2)}\right\}_{j \in
\mathrm{N}}$ intersecting $ A_2 C_{1} B_{j_{0}}^{(2)}$ is finite. Then we take a mutually disjoint
subsequence $\left\{B_{j_{\ell}}^{(2)}\right\}_{\ell \in \mathbb{N}}$ satisfying \eqref{mdel} and \eqref{intersection}. From
Lemma \ref{weightedestimate}, we can deduce that $[b,T]$ is not compact on $L_{\omega}^{p,
\kappa}\left(X\right)$. Thus, $b$ satisfies condition (ii) of Lemma \ref{lemvmo}.

\textbf{Case (iii)} Condition (iii) of Lemma \ref{lemvmo} does not hold for $b$. Then there exists $\gamma>0$ such that for any $r>0$, there exists $B \subset X \backslash B(x_0, r)$ with $M(b, B)>\gamma$. As in \cite{Pxq} for the $\gamma$ above, there exists a sequence $\left\{B_{j}^{(3)}\right\}_{j}$ of balls such that for any $j$,

\begin{equation}
 \quad \quad M\left(b, B_{j}^{(3)}\right)>\gamma,
\end{equation}
and for any $i \neq m$,

\begin{equation}
\gamma_{1} B_{i}^{(3)} \cap \gamma_{1} B_{m}^{(3)}=\emptyset,
\end{equation}
for sufficiently large $\gamma_{1}$
since, by Case (i) and (ii), $\left\{B_{j}^{(3)}\right\}_{j \in \mathbb{N}}$ satisfies the conditions (i) and (ii)
of Lemma \ref{lemvmo}, it follows that there exist positive constants $C_{\min }$ and $C_{\max }$ such that
\[
C_{\min } \leq r_{j} \leq C_{\max }, \quad \forall j \in \mathbb{N}.
\]
Using this and Lemma \ref{weightedestimate} we deduce that, if $[b, T]$ is compact on $L_{\omega}^{p,
\kappa}\left(X\right),$ then $b$ also satisfies condition (iii) of Lemma \ref{lemvmo}. This
completes the proof of Theorem \ref{thm main2}(ii) and hence of Theorem \ref{thm main2}.
\end{proof}

\section{Appendix: \  characterisation of ${\rm VMO}(X)$}

In this section, we provide the characterisation of VMO space on $X$
by giving the proof of Lemma \ref {lemvmo}.

\begin{proof}[Proof of Lemma \ref {lemvmo}]
In the following, for any integer $m$, we use $B^m$ to denote the ball $B(x_0, 2^m)$, where $x_0$ is a fixed point in $X$.

{\bf Necessary condition:} Assume that $f\in{\rm VMO}(X)$. If $f\in {\rm Lip}_{c}(\beta)$, then (i)-(iii) hold. In fact, by the uniform continuity, $f$ satisfies (i).  Since $f\in L^1(X)$,  $f$ satisfies (ii). By the fact that $f$ is compactly supported, $f$ satisfies (iii). If $f\in{\rm VMO}(X)\setminus {\rm Lip}_{c}(\beta)$, by definition, for any given $\varepsilon>0$, there exists $f_\varepsilon\in {\rm Lip}_{c}(\beta)$  such that $\|f-f_\varepsilon\|_{{\rm BMO}(X)}<\varepsilon$. Since $f_\varepsilon$ satisfies (i)-(iii), by the triangle inequality of ${\rm BMO}(X)$ norm, we can see (i)-(iii) hold for $f$.

{\bf Sufficient condition:}  In this proof for $j=1,2,\cdots, 8$,
 the value
$\alpha_j$ is a positive constant depending only on $n$ and $\alpha_i$ for $1\leq i<j$. Assume that $f\in {\rm BMO}(X)$ and satisfies (i)-(iii). To prove that $f\in{\rm VMO}(X)$, it suffices to show that there exist  positive constants $\alpha_1$,  $\alpha_2$
 such that, for any $\varepsilon>0$, there exists $\phi_\varepsilon\in {\rm BMO}(X)$ satisfying
\begin{align}\label{g-h}
\inf_{h\in {\rm Lip}_{c}(\beta)}\|\phi_\varepsilon-h\|_{{\rm BMO}(X)}<\alpha_1\varepsilon,
\end{align}
and
\begin{align}\label{g-f}
\|\phi_\varepsilon-f\|_{{\rm BMO}(X)}<\alpha_2\varepsilon.
\end{align}
By (i), there exist $i_\varepsilon\in\mathbb N$ such that
\begin{align}\label{ieps}
\sup\left\{M(f,B): r_B\leq 2^{-i_\varepsilon+4}\right\}<\varepsilon.
\end{align}
By (iii), there exists $j_\varepsilon\in\mathbb N$ such that
\begin{align}\label{jeps}
\sup\left\{M(f,B): B\cap B^{j_\varepsilon}=\emptyset\right\}<\varepsilon.
\end{align}

We first establish a cover of $X$.
 Observe that
\begin{align*}
B^{j_\varepsilon}=B^{-i_\varepsilon}\bigcup\left(\bigcup_{\nu=1}^{2^{j_\varepsilon+i_\varepsilon}-1}
B\left(x_0, (\nu+1)2^{-i_\varepsilon}\right)\setminus B\left(x_0, \nu2^{-i_\varepsilon}\right)
\right)=:\bigcup_{\nu=0}^{2^{j_\varepsilon+i_\varepsilon}-1}\mathcal R^{j_\varepsilon}_{\nu,-i_\varepsilon}
\end{align*}
For $m>j_\varepsilon$,
\begin{align*}\begin{split}
B^{m}\setminus B^{m-1}& =\bigcup_{\nu=0}^{2^{j_\varepsilon+i_\varepsilon-1}-1}
B\left(x_0, 2^{m-1}+ (\nu+1)2^{m-j_\varepsilon-i_\varepsilon}\right)\setminus B\left(x_0, 2^{m-1}+\nu2^{m-j_\varepsilon-i_\varepsilon}
\right)\\
&=:\bigcup_{\nu=0}^{2^{j_\varepsilon+i_\varepsilon-1}-1}\mathcal R^{m}_{\nu,m-j_\varepsilon-i_\varepsilon}.
\end{split}\end{align*}
For each $\mathcal R^{j_\varepsilon}_{\nu,-i_\varepsilon}$, $\nu=1,2,\cdots, 2^{j_\varepsilon+i_\varepsilon}-1$, let
$\tilde {\mathcal B}^{j_\varepsilon}_{\nu,-i_\varepsilon}$ be an open cover of $\mathcal R^{j_\varepsilon}_{\nu,-i_\varepsilon}$ consisting of open balls with radius $2^{-i_\varepsilon}$ and center on the sphere $S(x_0, (\nu+2^{-1})2^{-i_\varepsilon})$. Let $\mathcal B^{j_\varepsilon}_{0,-i_\varepsilon}=\{B(x_0,2^{-i_\varepsilon})\}$ and $\mathcal B^{j_\varepsilon}_{\nu,-i_\varepsilon}$ be the finite subcover of $\tilde {\mathcal B}^{j_\varepsilon}_{\nu,-i_\varepsilon}$. Similarly, for each $m>j_\varepsilon$ and $\nu=0,1,\cdots, 2^{j_\varepsilon+i_\varepsilon-1}-1$, let $\mathcal B^{m}_{\nu,m-j_\varepsilon-i_\varepsilon}$ be the finite cover of $\mathcal R^{m}_{\nu,m-j_\varepsilon-i_\varepsilon}$ consisting of open balls with radius $2^{m-j_\varepsilon-i_\varepsilon}$ and center on the sphere $S(x_0, (2^{m-1}+ (\nu+2^{-1})2^{m-j_\varepsilon-i_\varepsilon})$.

We define $B_x$ as follows.  If $x\in B^{j_\varepsilon}$, then there is $\nu\in\{0,1,\cdots, 2^{j_\varepsilon+i_\varepsilon}-1\}$ such that $x\in\mathcal R^{j_\varepsilon}_{\nu,-i_\varepsilon}$, let $B_x$ be a ball in $\mathcal B^{j_\varepsilon}_{\nu,-i_\varepsilon}$ that contains $x$. If $x\in B^m\setminus B^{m-1}$,  $m> j_\varepsilon $, then there is  $\nu\in\{ 0, 1, \cdots,  2^{j_\varepsilon+i_\varepsilon-1}-1\}$ such that $x\in \mathcal R^{m}_{\nu,m-j_\varepsilon-i_\varepsilon}$,  let $B_x$ be a ball in $\mathcal B^{m}_{\nu,m-j_\varepsilon-i_\varepsilon}$ that contains $x$.  We can see that if $\overline B_x\cap \overline B_{x'}\neq\emptyset$, then
\begin{align}\label{rbxy}
{\rm either}\ \  r_{B_x}\leq 2 ~ r_{B_{x'}}\ \  {\rm or}\ \  r_{B_{x'}}\leq 2~ r_{B_x}.
\end{align}
In fact, if $r_{B_x}>2r_{B_{x'}}$, then there is $m_0\in\mathbb N$ such that $x\in B^{m_0+2}\setminus B^{m_0+1}$ and $x'\in B^{m_0}$, thus
$$d(x,x')\geq d(x_0,x)-d(x_0,x')\geq2^{m_0+1}-2^{m_0}>2^{m_0+2-j_\varepsilon-i_\varepsilon}+2^{m_0-j_\varepsilon-i_\varepsilon}=r_{B_x}+r_{B_{x'}},$$
 which is contradict to the fact that $\overline B_x\cap \overline B_{x'}\neq\emptyset$ (Without loss of generality, here we assume that $A_0=1$ in the quasi-triangle inequality. Otherwise, we just need to take $r_{B^m}=([2A_0]+1)^m$ and make some modifications).

Now we define $\phi_\varepsilon$.
By (ii), there exists $m_\varepsilon>j_\varepsilon$ large enough such that when $r_B>2^{m_\varepsilon-i_\varepsilon-j_\varepsilon}$, we have
\begin{align}\label{lmeps}
M(f, B)<2^{n(-i_\varepsilon-j_\varepsilon-1)-1}\varepsilon.
\end{align}
Define
\begin{align*}
 \phi_\varepsilon(g)=
 \begin{cases}
f_{B_x}, &{\rm if}\ \  x\in B^{m_\varepsilon},\\
 f_{B^{m_\varepsilon}\setminus B^{m_\varepsilon-1}},  \quad &{\rm if}\ \ x\in X\setminus B^{m_\varepsilon}.
 \end{cases}
\end{align*}

We claim that there exists a positive constant $\alpha_3, \alpha_4$ such that if $\overline B_x\cap \overline B_{x'}\neq\emptyset$ or $x,x'\in X\setminus B^{m_\varepsilon-1}$, then
\begin{align}\label{gx-gy}
\left|\phi_{\varepsilon}(x)-\phi_{\varepsilon}(x') \right|<\alpha_3 \varepsilon.
\end{align}
And
if $2B_{x}\cap 2B_{x'}\neq \emptyset$, then for any $x_1\in B_{x}$, $x_2\in B_{x'}$, we have
\begin{align}\label{gx-gy1}
\left|\phi_{\varepsilon}(x_1)-\phi_{\varepsilon}(x_2) \right|<\alpha_4 \varepsilon.
\end{align}
Assume \eqref{gx-gy} and \eqref{gx-gy1} at the moment, we now continue to prove   the sufficiency  of   Lemma \ref {lemvmo}.

Now we show \eqref{g-h}. Let
$\tilde h_{\varepsilon}(x):=\phi_{\varepsilon}(x)-f_{B^{m_\varepsilon}\setminus B^{m_\varepsilon-1}}.$
By definition of $\phi_{\varepsilon}$, we can see that
$\tilde h_{\varepsilon}(x)=0$ for $x\in X\setminus B^{m_\varepsilon}$ and
$\|\tilde h_{\varepsilon}-\phi_{\varepsilon}\|_{{\rm BMO}(X)}=0.$

Observe that $\supp (\tilde h_\varepsilon)\subset B^{m_\varepsilon}$ and there exists a function $h_\varepsilon\in C_c(X)$ such that for any $x\in  X$,
$|\tilde h_{\varepsilon}(x)-h_{\varepsilon}(x)|<\varepsilon.$
Let $\eta(s)$ be an infinitely differentiable function defined on $[0,\infty)$ such that $0\leq \eta(s)\leq 1, \eta(s)=1$ for $0\leq s\leq 1$ and $\eta(s)=0$ for $s\geq2$. And let
$$\rho(x,y,t)=\Big(\int_X\eta(d(x,z)/t)d\mu(z)\Big)^{-1}\eta(d(x,z)/t)$$
and
$$h_\varepsilon^t(x)=\int_X\rho(x,y,t)h_\varepsilon(y)d\mu(y).$$
Then by \cite[Lemmas 3.15 and 3.23]{MS}, $h_\varepsilon^t(x)$ approaches to  $h_\varepsilon(x)$ uniformly for $x\in X$ as $t$ goes to $0$ and $h_\varepsilon^t\in {\rm Lip}_{c}(\beta)$ for $\beta>0$. Since
\begin{align*}
\|h_\varepsilon^t-\phi_\varepsilon\|_{{\rm BMO}(X)}
&\leq \|h_\varepsilon^t-h_\varepsilon\|_{{\rm BMO}(X)}
+\|h_\varepsilon-\tilde h_\varepsilon\|_{{\rm BMO}(X)}+
\| \tilde h_\varepsilon-\phi_\varepsilon\|_{{\rm BMO}(X)}\\
&\leq \|h_\varepsilon^t-h_\varepsilon\|_{{\rm BMO}(X)}+2\varepsilon,
\end{align*}
we can obtain \eqref{g-h} by letting $t$ go to $0$ and by taking $\alpha_1=2$.

Now we show \eqref{g-f}. To this end, we only need to prove that for any ball $B\subset X$,
$$M(f-\phi_\varepsilon, B)<\alpha_2 \varepsilon.$$
We first prove that for every $B_x$ with $x\in B^{m_\varepsilon}$,
\begin{equation}\label{Mfgx}
\int_{B_x}\left|f(x')-\phi_{\varepsilon}(x')\right|d\mu(x')\leq \alpha_5 \varepsilon \mu(B_x).
\end{equation}

In fact,
\begin{align*}
\int_{B_x}\left|f(x')-\phi_{\varepsilon}(x')\right|d\mu(x')
&=\int_{B_x\cap B^{m_\epsilon}}|f(x')-f_{B_{x'}}|d\mu(x')+ \int_{B_x\cap (X\setminus B^{m_\epsilon})}|f(x')-f_{B^{m_\varepsilon}\setminus B^{m_\varepsilon-1}}|d\mu(x').
\end{align*}

When 
 $x\in B(x_0, 2^{m_\varepsilon}-2^{m_\varepsilon-i_\varepsilon-j_\varepsilon})$, then
$B_x\subset B^{m_\epsilon}$, thus
\begin{align*}
\int_{B_x}\left|f(x')-\phi_{\varepsilon}(x')\right|d\mu(x')&=\int_{B_x}|f(x')-f_{B_{x'}}|d\mu(x')\\
&\leq\int_{B_x}|f(x')-f_{B_x}|d\mu(x')+\int_{B_x}|f_{B_x}-f_{B_{x'}}|d\mu(x')\\
&= \mu(B_x) M(f, B_x)+\int_{B_x}|f_{B_x}-f_{B_{x'}}|d\mu(x').
\end{align*}
Note that if $x'\in B_x$, then $B_x\cap B_{x'}\neq\emptyset$.
Therefore, If $B_x\cap B^{j_\varepsilon}=\emptyset$, by \eqref{jeps} and \eqref {gx-gy}, we have
$$\int_{B_x}\left|f(x')-\phi_{\varepsilon}(x')\right|d\mu(x')<(\varepsilon+\alpha_3\varepsilon)\mu(B_x).$$
If $B_x\cap B^{j_\varepsilon}\neq\emptyset$, then $r_{B_x}\leq 2^{-i_\varepsilon+1}$, then by \eqref{ieps} and \eqref{gx-gy},
$$\int_{B_x}\left|f(x')-\phi_{\varepsilon}(x')\right|d\mu(x')<(\varepsilon+\alpha_3\varepsilon)\mu(B_x).$$

When $x\in B^{m_\varepsilon}\setminus B(x_0, 2^{m_\varepsilon}-2^{m_\varepsilon-j_\varepsilon-i_\varepsilon})$, it is clear that $B_x\cap B^{j_\varepsilon}=\emptyset$, then by \eqref{jeps}, \eqref{lmeps} and \eqref{gx-gy}, we have
\begin{align*}
&\int_{B_x}\left|f(x')-\phi_{\varepsilon}(x')\right|d\mu(x')\\
&\leq\int_{B_x\cap B^{m_\epsilon}}|f(x')-f_{B_x}|d\mu(x')+ \int_{B_x\cap B^{m_\epsilon}}|f_{B_x}-f_{B_{x'}}|d\mu(x')\\
&\quad+\int_{B_x\cap (X\setminus B^{m_\epsilon})}|f(x')-f_{B^{m_\varepsilon+1}}|d\mu(x')+\int_{B_x\cap (X\setminus B^{m_\epsilon})}|f_{B^{m_\varepsilon+1}}-f_{B^{m_\varepsilon}\setminus B^{m_\varepsilon-1}}|d\mu(x')\\
&\leq \mu(B_x)M(f,B_x)+\alpha_3\varepsilon\mu(B_x)+\mu(B^{m_\varepsilon+1})M(f,B^{m_\varepsilon+1})
+{\mu(B^{m_\varepsilon+1}) \mu(B_x)\over \mu( B^{m_\varepsilon}\setminus B^{m_\varepsilon-1})}M(f,B^{m_\varepsilon+1})\\
&<(C_1\varepsilon+\alpha_3\varepsilon)\mu(B_x).
\end{align*}
Then \eqref {Mfgx} holds by taking $\alpha_5=(C_1+\alpha_3)$.

Let $B$ be an arbitrary ball in $X$, then
$M(f-\phi_\varepsilon, B)\leq M(f, B) + M(\phi_\varepsilon, B). $
If $B\subset B^{m_\varepsilon}$ and $\max\{r_{B_x}: B_x\cap B\neq\emptyset\}>8r_B$, then
\begin{align}\label{min}
\min\{r_{B_x}: B_x\cap B\neq \emptyset\}>2r_B.
\end{align}

In fact, assume that $r_{B_{\widehat{x}}}=\max\{r_{B_x}: B_x\cap B\neq\emptyset\}$ and $\widehat{x}\in B^{l_0}\setminus B^{l_0-1}$ for some $l_0\in\mathbb Z$.
Then $B\subset B^{l_0}\cap {3\over 2}B_{\widehat{x}}$.
If $l_0\leq j_\varepsilon$, then \eqref{min} holds.
If $l_0> j_\varepsilon$, then $r_{B_{\widehat{x}}}=2^{l_0-j_\varepsilon-i_\varepsilon}$, and
$$r_B<{1\over 8}r_{B_{\widehat{x}}}=2^{l_0-j_\varepsilon-i_\varepsilon-3}.$$
Since for any $x'\in{3\over 2}B_{\widehat{x}}$,
\begin{align*}
d(x_0,x')&\geq d(x_0,\widehat{x})-d(\widehat{x}, x')\geq2^{l_0-1}-{3\over 2}2^{l_0-j_\varepsilon-i_\varepsilon}
>2^{l_0-1}-2^{l_0-j_\varepsilon-i_\varepsilon+1},
\end{align*}
we have
$$\operatorname{dist}(x_0,{3\over 2}B_{\widehat{x}}):=\inf_{x'\in {3\over 2}B_{\widehat{x}}}d(x_0,x')>2^{l_0-1}-2^{l_0-j_\varepsilon-i_\varepsilon+1}.$$
Thus $B\subset B^{l_0}\setminus {3\over 2}B^{l_0-2}$. Therefore, if $B_x\cap B\neq\emptyset$, then $x\in B^{l_0}\setminus B^{l_0-2}$, which implies that
$r_{B_x}\geq2^{l_0-2-j_\varepsilon-i_\varepsilon}>2r_B.$

From \eqref{min} we can see that  if $B_{x_i}\cap B\neq \emptyset$ and $B_{x_j}\cap B\neq \emptyset$, then $2B_{x_i}\cap 2B_{x_j}\neq \emptyset$.
Then by \eqref{gx-gy1}, we can get
\begin{align*}
M(\phi_\varepsilon, B)
&\leq{1\over \mu(B)}\int_B {1\over \mu(B)}\int_B \left| \phi_\varepsilon(x)- \phi_\varepsilon(x') \right|d\mu(x')d\mu(x)\\
&= {1\over \mu(B)^2}\sum_{i:B_{x_i}\cap B\neq\emptyset}\int_{B_{x_i}\cap B}
\sum_{j:B_{x_j}\cap B\neq\emptyset}\int_{B_{x_j}\cap B}\left| \phi_\varepsilon(x)- \phi_\varepsilon(x') \right|d\mu(x')d\mu(x)\\
&<\alpha_4\varepsilon{1\over \mu(B)^2}\left(\sum_{i:B_{x_i}\cap B\neq\emptyset}\mu(B_{x_i}\cap B)\right)
\left(\sum_{i:B_{x_j}\cap B\neq\emptyset}\mu(B_{x_j}\cap B)\right)<\alpha_4\alpha_6^2\varepsilon.
\end{align*}
Moreover, if $B\cap B^{j_\epsilon}\neq\emptyset$, then by \eqref{min}, $r_B<2^{-i_\varepsilon}$, thus by \eqref{ieps}, we have
$M(f, B)<\varepsilon.$
If $B\cap B^{j_\epsilon}=\emptyset$, then by \eqref{jeps}, $M(f, B)<\varepsilon.$
Consequently,
$$M(f-\phi_\varepsilon, B)\leq M(f, B) + M(\phi_\varepsilon, B)<\left(1+\alpha_4\alpha_6^2\right)\varepsilon.$$

If $B\subset B^{m_\varepsilon}$ and $\max\{r_{B_x}: B_x\cap B\neq\emptyset\}\leq 8r_B$,
since the number of $B_x$ with $x\in B^{m_\varepsilon}$ that covers $B$ is bounded by $\alpha_7$,
 by  \eqref{Mfgx},  we have
\begin{align*}
M(f-\phi_\varepsilon, B)
&\leq {2\over \mu(B)}\int_B \left|f(x)-\phi_\varepsilon(x)\right|d\mu(x)\leq {2\over \mu(B)} \sum_{i:B_{x_i}\cap B\neq\emptyset}\int_{B_{x_i}} \left|f(x)-\phi_\varepsilon(x)\right|d\mu(x)\\
&\leq {2\over \mu(B)}\alpha_5\varepsilon \sum_{i:B_{x_i}\cap B\neq\emptyset}\mu(B_{x_i})\leq {2\over \mu(B)} \alpha_5\alpha_7\varepsilon\mu(8B)\leq C_2\alpha_5\alpha_7\varepsilon.
\end{align*}

If $B\subset X\setminus B^{m_\varepsilon-1}$, then $B\cap B^{j_\varepsilon}=\emptyset$, from \eqref{jeps} we can see $M(f, B)<\varepsilon$. By \eqref{gx-gy},
\begin{align*}
M(\phi_\varepsilon, B)
\leq{1\over \mu(B)}\int_B {1\over \mu(B)}\int_B \left|\phi_\varepsilon(x)-\phi_\varepsilon(x') \right|d\mu(x')d\mu(x)
<\alpha_3\varepsilon.
\end{align*}
Therefore,
$$M(f-\phi_\varepsilon, B)\leq M(f, B)+M(\phi_\varepsilon, B)<(1+\alpha_3)\varepsilon.$$

If $B\cap (X\setminus B^{m_\varepsilon})\neq \emptyset$ and $B\cap B^{m_\varepsilon-1}\neq \emptyset$. Let $p_{\scriptscriptstyle B}$ be the smallest integer such that $B\subset B^{p_{\scriptscriptstyle B}}$, then $p_{\scriptscriptstyle B}>m_\varepsilon$.
If $p_{\scriptscriptstyle B}=m_\varepsilon+1$, then
$r_B>{1\over 2}(2^{m_\varepsilon}-2^{m_\varepsilon-1})=2^{m_\varepsilon-2}$.
If $p_{\scriptscriptstyle B}>m_\varepsilon+1$, then
$r_B>{1\over 2}(2^{p_{\scriptscriptstyle B}-1}-2^{m_\varepsilon-1})$. Thus
$${\mu(B^{p_{\scriptscriptstyle B}})\over \mu(B)}\leq {C_3\over2}.$$
Therefore,
\begin{align*}
M(f-\phi_\varepsilon, B)
&\leq{1\over \mu(B)}\int_B \left|f(x)-\phi_\varepsilon(x)-(f-\phi_\varepsilon)_{B^{p_{\scriptscriptstyle B}}} \right|d\mu(x)
+\left| (f-\phi_\varepsilon)_{B^{p_{\scriptscriptstyle B}}}-(f-\phi_\varepsilon)_{B}\right|\\
&\leq2{\mu(B^{p_{\scriptscriptstyle B}})\over \mu(B)}{1\over \mu(B^{p_{\scriptscriptstyle B}})}\int_{B^{p_{\scriptscriptstyle B}}} \left|f(x)-\phi_\varepsilon(x)-(f-\phi_\varepsilon)_{B^{p_{\scriptscriptstyle B}}} \right|d\mu(x)\\
&\leq C_3\left(M(f, B^{p_{\scriptscriptstyle B}})+M(\phi_\varepsilon, B^{p_{\scriptscriptstyle B}}) \right)\leq C_3\left(\varepsilon+M(\phi_\varepsilon, B^{p_{\scriptscriptstyle B}}) \right),
\end{align*}
where the last inequality comes from \eqref{lmeps}.
By definition,
\begin{align*}
M(\phi_\varepsilon, B^{p_{\scriptscriptstyle B}})
&\leq {1\over \mu(B^{p_{\scriptscriptstyle B}})}\int_{B^{p_{\scriptscriptstyle B}}} \left|\phi_\varepsilon(x)-(\phi_\varepsilon)_{B^{p_{\scriptscriptstyle B}}\setminus B^{m_\varepsilon}} \right|d\mu(x)
+\left|(\phi_\varepsilon)_{B^{p_{\scriptscriptstyle B}}\setminus B^{m_\varepsilon}}- (\phi_\varepsilon)_{B^{p_{\scriptscriptstyle B}}}\right|\\
&\leq{2\over \mu(B^{p_{\scriptscriptstyle B}})}\int_{B^{p_{\scriptscriptstyle B}}} \left|\phi_\varepsilon(x)-(\phi_\varepsilon)_{B^{p_{\scriptscriptstyle B}}\setminus B^{m_\varepsilon}} \right|d\mu(x).
\end{align*}
By \eqref{jeps}, \eqref{Mfgx} and the fact that $\phi_\varepsilon(x)=f_{B^{m_\varepsilon}\setminus B^{m_\varepsilon-1}}$ if $x\in X\setminus B^{m_\varepsilon}$, we have
\begin{align*}
&\int_{B^{p_{\scriptscriptstyle B}}} \left|\phi_\varepsilon(x)-(\phi_\varepsilon)_{B^{p_{\scriptscriptstyle B}}\setminus B^{m_\varepsilon}} \right|d\mu(x)\leq \int_{B^{p_{\scriptscriptstyle B}}}{1\over \mu(B^{p_{\scriptscriptstyle B}}\setminus B^{m_\varepsilon})}\int_{B^{p_{\scriptscriptstyle B}}\setminus B^{m_\varepsilon}}
|\phi_\varepsilon(x)-\phi_\varepsilon(x')|d\mu(x')d\mu(x)\\
&= \int_{B^{m_\varepsilon}}
|\phi_\varepsilon(x)-f_{B^{m_\varepsilon}\setminus B^{m_\varepsilon-1}}|d\mu(x)\\
&\leq \int_{B^{m_\varepsilon}}
|\phi_\varepsilon(x)-f(x)|d\mu(x)+\int_{B^{m_\varepsilon}}
|f(x)-f_{B^{m_\varepsilon}}|dx+\mu(B^{m_\varepsilon})|f_{B^{m_\varepsilon}}-f_{B^{m_\varepsilon}\setminus B^{m_\varepsilon-1}}|\\
&\leq \sum_{i:B_{x_i}\cap B^{m_\varepsilon}\neq\emptyset, x_i\in B^{m_\varepsilon}}
\int_{B_{x_i}}
|\phi_\varepsilon(x)-f(x)|d\mu(x)+\left(\mu(B^{m_\varepsilon})+{\mu(B^{m_\varepsilon})^2\over \mu(B^{m_\varepsilon}\setminus B^{m_\varepsilon-1}) }\right)M(f, B^{m_\varepsilon})\\
&<\alpha_5\varepsilon\sum_{i:B_{x_i}\cap B^{m_\varepsilon}\neq\emptyset, x_i\in B^{m_\varepsilon}}\mu(B_{x_i})+3\varepsilon\mu(B^{m_\varepsilon})<(\alpha_5\alpha_8+3)\varepsilon\mu(B^{m_\varepsilon}).
\end{align*}
Therefore,
\begin{align*}
M(f-\phi_\varepsilon, B)&\leq C_3\left(\varepsilon+M(\phi_\varepsilon, B^{p_{\scriptscriptstyle B}}) \right)\leq C_3\left(\varepsilon+{2\mu(B^{m_\varepsilon})\over \mu(B^{p_{\scriptscriptstyle B}})} (\alpha_5\alpha_8+3)\varepsilon\right)\\
&<C_4\left(  \alpha_5\alpha_8+3\right)\varepsilon.
\end{align*}
Then \eqref{g-f} holds by taking $\alpha_2=\max\{1+\alpha_4\alpha_6^2, 1+\alpha_3, C_4\left(  \alpha_5\alpha_8+3\right)\}$. This finishes the proof of Lemma \ref {lemvmo}.

%

\medskip
{\it Proof of \eqref{gx-gy}:}
\medskip


 We first claim that
\begin{align}\label{tgx-gy}
\sup\left\{\left|f_{B_x}-f_{B_{x'}}\right|: x, x'\in B^{m_\varepsilon}\setminus B^{m_\varepsilon-1} \right\}<C_5\varepsilon.
\end{align}


By \eqref{lmeps}, for any $x\in B^{m_\varepsilon}\setminus B^{m_\varepsilon-1}$, we have
\begin{align*}
\left| f_{B_x}-f_{B^{m_{\varepsilon}+1}}\right|
&\leq
{\mu(B^{m_{\varepsilon}+1})\over \mu(B_x)}{1\over \mu(B^{m_{\varepsilon}+1})}\int_{B^{m_{\varepsilon}+1}}\left|f(x')-f_{B^{m_{\varepsilon}+1}}\right|d\mu(x')\\
&={\mu(B^{m_{\varepsilon}+1})\over \mu(B_x)}M(f, B^{m_\varepsilon +1})
<{C_5\over 2}\varepsilon.
\end{align*}
Similarly, for any $x'\in B^{m_\varepsilon}\setminus B^{m_\varepsilon-1}$,
$|f_{B_{x'}}-f_{B^{m_{\varepsilon}+1}}|<{C_5\over 2}\varepsilon.$
Consequently, \eqref{tgx-gy} holds.

For the case $x, x'\in X\setminus B^{m_\varepsilon-1}$, firstly, if $x, x'\in X\setminus B^{m_\varepsilon}$, then by definition
$$\left|\phi_{\varepsilon}(x)-\phi_{\varepsilon}(x') \right|=0.$$
Secondly, if $x,x'\in B^{m_\varepsilon}\setminus B^{m_\varepsilon-1}$, then by \eqref{tgx-gy}, we have
$$\left|\phi_{\varepsilon}(x)-\phi_{\varepsilon}(x') \right|<C_5\varepsilon.$$
Thirdly, without loss of generality, we may assume that $x\in B^{m_\varepsilon}\setminus B^{m_\varepsilon-1}$ and $x'\in X\setminus B^{m_\varepsilon}$, then by \eqref{lmeps}, we have
\begin{align*}
\left|\phi_{\varepsilon}(x)-\phi_{\varepsilon}(x') \right|
&=\left|f_{B_x}-f_{B^{m_\varepsilon}\setminus B^{m_\varepsilon-1}}  \right|\leq \left|f_{B_x}-f_{B^{m_\varepsilon+1}}  \right|+\left|f_{B^{m_\varepsilon+1}}-f_{B^{m_\varepsilon}\setminus B^{m_\varepsilon-1}}  \right|\\
&\leq {\mu(B^{m_\varepsilon+1})\over \mu(B_x)} M(f, B_{m_{\varepsilon+1}})+{\mu(B^{m_\varepsilon+1})\over \mu(B^{m_\varepsilon}\setminus B^{m_\varepsilon-1})}M(f, B_{m_{\varepsilon+1}})\\
&\leq \left({\mu(B^{m_\varepsilon+1})\over \mu(B_x)} +{\mu(B^{m_\varepsilon+1})\over \mu(B^{m_\varepsilon}\setminus B^{m_\varepsilon-1})}\right)M(f, B_{m_{\varepsilon+1}})\\
&<C_6\varepsilon.
\end{align*}

For the case $\overline B_x\cap \overline B_{x'}\neq\emptyset$ and  $x,x'\in B^{m_\varepsilon-1}$, we may assume $B_x\neq B_{x'}$ and $r_{B_x}\leq r_{B_{x'}}$. By \eqref{rbxy},
 $B_{x'}\subset 5B_x\subset 15B_{x'}$.
 If $x'\in B^{j_\varepsilon+1}$, then by \eqref{ieps}, we have
\begin{align*}
\left|\phi_{\varepsilon}(x)-\phi_{\varepsilon}(x') \right|
&=\left|f_{B_x}-f_{B_{x'}}  \right|\leq \left|f_{B_x}-f_{3B_{x'}}  \right|+\left|f_{B_{x'}}-f_{3B_{x'}}  \right|\\
&\leq\left( {\mu(3B_{x'})\over \mu(B_x)} +{\mu(3B_{x'})\over \mu(B_{x'})}\right)M(f, 3B_{x'}) \\
&\leq C_7\varepsilon.
\end{align*}
If $x'\notin B^{j_\varepsilon+1}$, then $3B_{x'}\cap B^{j_\varepsilon}=\emptyset$, by \eqref{jeps}, we have
$$\left|\phi_{\varepsilon}(x)-\phi_{\varepsilon}(x') \right|\leq C_7M(f, 3B_{x'})\leq C_7\varepsilon.$$
Therefore, \eqref{gx-gy} holds by taking $\alpha_3=\max\{C_5,C_6,C_7\}$.

\medskip
{\it Proof of \eqref{gx-gy1}:}
\medskip


Since $x_1\in B_{x}$, $x_2\in B_{x'}$, we have $B_{x_1}\cap B_x\neq\emptyset$ and $B_{x_2}\cap B_{x'}\neq\emptyset$, by \eqref{gx-gy},
\begin{align*}
\left|\phi_{\varepsilon}(x_1)-\phi_{\varepsilon}(x_2) \right|
&\leq \left|\phi_{\varepsilon}(x_1)-\phi_{\varepsilon}(x) \right|+\left|\phi_{\varepsilon}(x)-\phi_{\varepsilon}(x') \right|+\left|\phi_{\varepsilon}(x')-\phi_{\varepsilon}(x_2) \right|\\
&\leq 2\alpha_3\varepsilon+\left|\phi_{\varepsilon}(x)-\phi_{\varepsilon}(x') \right|.
\end{align*}
We may assume $B_x\neq B_{x'}$ and $r_{B_x}\leq r_{B_{x'}}$.
If $x,x'\in X\setminus B^{m_\varepsilon-1}$, then \eqref{gx-gy1} follows from \eqref{gx-gy}.
If $x,x'\in\ B^{m_\varepsilon-1}$,
when $ x'\in B^{j_\varepsilon+1}$, then $2^{-i_\varepsilon}\leq r_{B_x}\leq r_{B_{x'}}\leq 2^{-i_\varepsilon+1}$, thus  $B_{x'}\subset 10B_x\subset 60B_{x'}$, by \eqref{ieps}, we have
\begin{align*}
\left|\phi_{\varepsilon}(x)-\phi_{\varepsilon}(x') \right|
&\leq \left|f_{B_x}-f_{6B_{x'}}  \right|+\left|f_{B_{x'}}-f_{6B_{x'}}  \right|=\left( {\mu(6B_{x'})\over \mu(B_x)} +{\mu(6B_{x'})\over \mu(B_{x'})}\right)M(f, 6B_{x'})\\
&\leq C_9\varepsilon.
\end{align*}
When $x'\notin B^{j_\varepsilon+1}$, then there exist $\tilde m_0\in \mathbb N$ and $\tilde m_0\geq j_\varepsilon+2$ such that $x'\in B^{\tilde m_0}\setminus B^{\tilde m_0-1}$. Since $2B_{x}\cap 2B_{x'}\neq \emptyset$, we have $B_x\subset 6B_{x'}$.
Note that $6B_{x'}\cap B^{\tilde m_0-2}=\emptyset$, (in fact, for any $\tilde x\in 6B_{x'}$, $d(x_0,\tilde x)\geq d(x_0, x')-d(x',\tilde x)\geq 2^{\tilde m_0-1}-6\cdot2^{\tilde m_0-j_\varepsilon-i_\varepsilon}>2^{\tilde m_0-2}$), thus $B_x\cap B^{\tilde m_0-2}=\emptyset$ and then ${1\over 2}r_{B_{x'}}=2^{\tilde m_0-1-j_\varepsilon-i_\varepsilon}\leq r_{B_x}\leq 2^{\tilde m_0-j_\varepsilon-i_\varepsilon}=r_{B_{x'}}$. Therefore, $B_{x'}\subset 10 B_x$. Then by \eqref{jeps}, we have
$$\left|\phi_{\varepsilon}(x)-\phi_{\varepsilon}(x') \right|\leq C_9M(f, 6B_{x'})<C_9\varepsilon.$$
If $x\in B^{m_\varepsilon-1}$ and $x'\in X\setminus B^{m_\varepsilon-1}$, since $2B_{x}\cap 2B_{x'}\neq \emptyset$, by the construction of $B_x$ we can see that $x\in B^{m_\varepsilon-1}\setminus B^{m_\varepsilon-2}$ and $x'\in B^{m_\varepsilon}\setminus B^{m_\varepsilon-1}$. Thus, $B_{x'}\subset 10 B_{x}\subset 40 B_{x'}$. Then by \eqref{lmeps}, we have
$$\left|\phi_{\varepsilon}(x)-\phi_{\varepsilon}(x') \right|  <C_{10}M(f, 4B_{x'})<C_{10}\varepsilon.$$
Taking $\alpha_4=C_{9}+C_{10}+2\alpha_3$, then \eqref{gx-gy1} holds.
%
%
%
%
\end{proof}

\bigskip
{\bf Acknowledgement:} R.M. Gong is supported by the State Scholarship Fund of China (No. 201908440061). J. Li is supported by ARC DP 170101060.

\bibliographystyle{amsplain}

\medskip

Ruming Gong, School of Mathematical Sciences, Guangzhou University, Guangzhou, China.

\smallskip

{\it E-mail}: \texttt{gongruming@gzhu.edu.cn}

\vspace{0.3cm}

Ji Li, Department of Mathematics, Macquarie University, NSW, 2109, Australia.

\smallskip

{\it E-mail}: \texttt{ji.li@mq.edu.au}

\vspace{0.3cm}

Elodie Pozzi, Department of Mathematics and Statistics,
         Saint Louis University,
         220 N. Grand Blvd, 63103 St Louis MO, USA.

\smallskip

{\it E-mail}: \texttt{elodie.pozzi@slu.edu}

\vspace{0.3cm}

Manasa N. Vempati,
Department of Mathematics,
Washington University -- St. Louis,
One Brookings Drive,
St. Louis, MO USA 63130-4899

\smallskip
{\it E-mail}: \texttt{m.vempati@wustl.edu}

\end{document}